\documentclass[12pt,a4paper,reqno]{amsart} 
\usepackage[left=1.9cm,right=1.9cm,top=2.8cm,bottom=2.8cm,a4paper]{geometry}
\usepackage{amsmath,amscd,amssymb,latexsym}
\usepackage{graphicx}
\usepackage{enumerate}
\usepackage{stackengine}
\stackMath   
\usepackage{placeins}
\usepackage{appendix}
\usepackage{mathrsfs} 
\usepackage{accents}

\makeatletter \@namedef{subjclassname@2020}{
  \textup{2020} Mathematics Subject Classification}
\makeatother

\usepackage{color}
\definecolor{gray}{rgb}{0.65,0.65,0.65}
\definecolor{ultramarine}{rgb}{0.07, 0.04, 0.56}
\usepackage[pagebackref=false,colorlinks=true,linkcolor=ultramarine,citecolor=ultramarine,urlcolor=cyan]{hyperref}

\begin{document}
\setlength{\unitlength}{1mm}

\newtheorem{thm}{Theorem}[section]
\newtheorem{lem}[thm]{Lemma}
\newtheorem{prop}[thm]{Proposition}
\newtheorem{defn}[thm]{Definition}
\newtheorem{cor}[thm]{Corollary}
\newtheorem{rmk}[thm]{Remark}
\newtheorem{exmp}[thm]{Example}
\newtheorem{fact}[thm]{Fact}
\newtheorem{conj}[thm]{Conjecture}

\newcommand{\QQ}{\mathbb{Q}}
\newcommand{\ZZ}{\mathbb{Z}}
\newcommand{\RR}{\mathbb{R}}
\newcommand{\NN}{\mathbb{N}}
\newcommand{\XX}{\mathbb{X}}
\newcommand{\FF}{\mathbb{F}}
\newcommand{\PP}{\mathbb{P}}

\newcommand{\cO}{\mathcal{O}}
\newcommand{\cN}{\mathcal{N}}
\newcommand{\cA}{\mathcal{A}}
\newcommand{\cS}{\mathcal{S}}

\newcommand{\fA}{\mathfrak{A}}
\newcommand{\fa}{\mathfrak{a}}
\newcommand{\fb}{\mathfrak{b}}
\newcommand{\fq}{\mathfrak{q}}
\newcommand{\fN}{\mathfrak{N}} 
\newcommand{\fM}{\mathfrak{M}}
\newcommand{\fP}{\mathfrak{P}}

\newcommand{\sgn}{\mathrm{sgn\,}}
\newcommand{\chisgn}{\chi_{\rm sgn}}

\newcommand{\bA}{\mathbf{A}}
\newcommand{\sA}{\mathscr{A}}
\newcommand{\sS}{\mathscr{S}}
\newcommand{\aA}{\mathbb{A}}

\title[Genus character $L$-functions of orders in quadratic function fields]
{Genus character $L$-functions of orders in quadratic function fields}

\author{Jigu Kim and Yoshinori Mizuno}

\address{Department of Mathematics, Jeonbuk National University, Jeonju, 54896, Republic of Korea}

\email{jigu.kim.math@gmail.com}

\address{Department of Mathematics, Nagoya Institute of Technology, Nagoya, 466-8555, Japan}

\email{mizuno.yoshinori@nitech.ac.jp}


\maketitle

\noindent{\bf Abstract.} 
We define genus character $L$-functions of orders in quadratic function fields of odd characteristic. We express each as a product of two quadratic $L$-functions and explicit local factors at primes dividing the conductor, with a possible shift of the variable in one of the two quadratic $L$-functions. This product formula yields  meromorphic continuation and functional equations. 
In the real case, we evaluate special values in terms of continued fractions and derive formulas for the product of class numbers of two imaginary quadratic function fields. These formulas are function field analogues of the Kaneko--Mizuno--Zagier formula. 
\medskip

\section{Introduction}\label{intro-sec}
Genus characters on narrow class groups of maximal orders in quadratic number fields can be defined in terms of Kronecker characters and ideal norms, and their associated $L$-functions factor as products of two Dirichlet $L$-functions (cf. \cite[\S 12]{Zag81}).
For arbitrary orders, the second named author and Kaneko \cite{KM20} defined general genus characters via primitive binary quadratic forms and established a product formula with explicit local factors for the associated $L$-functions. 
Ibukiyama \cite[Theorem 4.3]{Ibu26} gave an independent proof of this product formula using an adelic description of genus characters. 
We first recall this formula, together with the resulting meromorphic continuation and functional equation (cf. \cite[Theorem 1 and the discussion on p. 71]{KM20}). 
\begin{thm}\label{KM-L-function-formula}
Let $\Delta=d_1d_2f_0^2$ be a nonsquare integer, where $d_1$ and $d_2$ are distinct fundamental discriminants (with $1$ allowed) and $f_0\in\NN$. 
Let $\chi_{d_1,d_2}^{(\Delta)}$ be the genus character associated to the quadratic order $\cO_\Delta$ of discriminant $\Delta$, and let $m_p$ be the integer such that $p^{m_p}\parallel f_0$. Then for $\mathrm{Re}(s)>1$, the genus character $L$-function satisfies 
\begin{align*}
&L(s,\chi_{d_1,d_2}^{(\Delta)})
:=\sum_{\cO_\Delta\text{-invertible ideal }\fa\subset\cO_\Delta}
\frac{\chi_{d_1,d_2}^{(\Delta)}(\fa)}{\fN_\Delta(\fa)^s}\\
=&L(s,\chi_{d_1})L(s,\chi_{d_2})\\
&\times\prod_{\substack{p\mid f_0\\p\mathrm{:prime}}}
\frac{(1-\chi_{d_1}(p)p^{-s})(1-\chi_{d_2}(p)p^{-s})
-p^{m_p-1-2m_ps}(p^{1-s}-\chi_{d_1}(p))(p^{1-s}-\chi_{d_2}(p))}{1-p^{1-2s}}.
\end{align*}
Here $\fN_\Delta(\fa)=(\cO_\Delta:\fa)$ is the ideal norm, and $L(s,\chi_{d})$ is the Dirichlet $L$-function of the Kronecker character $\chi_{d}(n)=\big(\frac{d}{n}\big)$.
Furthermore, put
$$\Lambda(s,\chi_{d_1,d_2}^{(\Delta)})
:=f_0^sG_{d_1}(s)G_{d_2}(s)L(s,\chi_{d_1,d_2}^{(\Delta)}),$$
where 
$$G_d(s):=\left(\frac{|d|}{\pi}\right)^{s/2}
\Gamma\left(\frac{s+\delta_d}{2}\right),\qquad
\delta_d=\left\{
\begin{array}{ll}
0 &\text{if }d>0,\\
1 &\text{if }d<0.\end{array}\right.$$
Then $L(s,\chi_{d_1,d_2}^{(\Delta)})$ admits meromorphic continuation to the whole complex plane and satisfies the functional equation
$$\Lambda(s,\chi_{d_1,d_2}^{(\Delta)})=\Lambda(1-s,\chi_{d_1,d_2}^{(\Delta)}).$$
\end{thm}

Zagier \cite{Zag75,Zag81} used the theory of continued fractions to obtain a formula for the product of class numbers of two imaginary quadratic number fields with relatively prime fundamental discriminants.
Combining the product formula in Theorem \ref{KM-L-function-formula} with a continued fraction expression for the special value at $s=0$, the second named author and Kaneko generalized Zagier's formula to the case of two distinct negative fundamental discriminants with a general conductor (cf. \cite[Theorem 3]{KM20}). 
\begin{thm}\label{KMZ-formula}
With the notation of Theorem \ref{KM-L-function-formula}, assume that $d_1$ and $d_2$ are negative. Let $\XX_\Delta^0$ be the set of all reduced quadratic irrationals of discriminant $\Delta$. 
For $\alpha\in\XX_\Delta^0$, the value $\chi_{d_1,d_2}^{(\Delta)}(\alpha)$ is understood via the corresponding ideal, and $\lfloor\alpha\rfloor$ denotes the greatest integer less than or equal to $\alpha$. Then
\begin{align*}
&24\frac{h(d_1)h(d_2)}{\omega(d_1)\omega(d_2)}
\prod_{\substack{p\mid f_0\\p\mathrm{:prime}}}
\frac{(1-\chi_{d_1}(p))(1-\chi_{d_2}(p))-p^{m_p-1}(p-\chi_{d_1}(p))(p-\chi_{d_2}(p))}{1-p}\\
&\quad=\sum_{\alpha\in\XX_\Delta^0}
\chi_{d_1,d_2}^{(\Delta)}(\alpha)\lfloor\alpha\rfloor.
\end{align*}
Here $h(d)$ and $\omega(d)$ denote the class number and the number of units in the ring of integers of $\QQ(\sqrt{d})$, respectively. 
\end{thm}
We remark that the last summation in Theorem \ref{KMZ-formula} can be expressed 
as alternating sums of all elements in periods of continued fractions of representatives of equivalence classes of $\XX_\Delta^0$ (cf. \cite[p. 88]{KM20}).

We consider a quadratic function field of the form $\FF_q(T,\sqrt{D(T)})$, 
where $q$ is a power of an odd prime and $D(T)$ is a square-free polynomial with coefficients in $\FF_q$.  We call $D(T)$ the {\it fundamental discriminant} of the field. 
Hayes \cite{Hay87} proved function field analogues of Theorem \ref{KMZ-formula} for the cases where one of the two fundamental discriminants is given by $g$, $T$ or $gT$, and $g$ is a generator of $\FF_q^\times$; Gonz\'{a}lez \cite{Gon92} obtained partial results by another method. 
In the case of relatively prime fundamental discriminants of two quadratic function fields with $q=p$, 
Richter \cite{Rit00} introduced genus characters and extended these formulas to a generality comparable to Zagier’s.
We refer to \cite{CY08} for similar results for quadratic function fields of even characteristic.

In this paper, we prove analogues of Theorems \ref{KM-L-function-formula} and \ref{KMZ-formula}, which are stated in Theorems \ref{KM-thm-1-Rit-12.3.2} and \ref{KM-thm-3-Rit-15.1.2}, respectively.
Let $\FF_q$ be the finite field of order $q$ with characteristic $p\neq2$, and let $g$ be a fixed generator of $\FF_q^\times$. 
Let $k=\FF_q(T)$ and $\cA=\FF_q[T]$, where $T$ is an indeterminate. 
We denote by $\cA^{(1)}$ the subset of $\cA$ consisting of monic elements.
By a prime in $\cA$ we always mean a nonconstant monic irreducible polynomial.
For the relation $A\mid B$, we always assume that $A\in\cA^{(1)}$, $B\in\cA$ and $A$ divides $B$. 
Let $k_\infty$ be the completion of $k$ at $\infty$. We denote the polynomial part of $Z=\sum_{i=-\infty}^{d}a_i T^i\in k_\infty$ by $\lfloor Z\rfloor=\sum_{i=0}^{d}a_i T^i$. 
For $Z=\sum_{i=-\infty}^{d}a_i T^i\in k_\infty^\times$ with $a_d\neq0$, we put $\deg{Z}=d$, $\|Z\|=q^{\deg{Z}}$, and $\sgn{Z}=a_d$. 
We also put $\deg{0}=-\infty$ and $\|0\|=\sgn{0}=0$. 
When $\Delta$ is a monic polynomial of even degree, we always choose $\sqrt{\Delta}\in k_\infty$ with $\sgn(\sqrt{\Delta})=1$ and regard $k(\sqrt{\Delta})$ as a subfield of $k_\infty$.

Let $\chisgn:k_\infty \to \{0,\pm1\}$ be the quadratic sign character at infinity, that is, the quadratic character of $\FF_q$ evaluated on the leading coefficient:
$$\chisgn(Z)=\left\{\begin{array}{rl}
1 &\text{if }\,\sgn{Z}\in {\FF_q^\times}^{2},\\ 
-1 &\text{if }\,\sgn{Z}\not\in\FF_q^{2},\\
0 &\text{if }\,Z=0.
\end{array}\right.$$
We use the natural action of $\mathrm{GL}_2(\cA)$ on the projective line 
$\PP^1=k_\infty\cup\{P_\infty\}$ and standard continued fractions of real quadratic irrational functions along with their complete quotients, partial quotients, and quasi-periods. These conventions are collected in Appendix \ref{app:continued-fractions}.

The ideal-theoretic notions of quadratic orders and Jacobi symbols are recalled in Section \ref{preliminary-sec}. 
For $\Delta\in\cA\setminus \cA^2$, let $\cO_\Delta=\cA+\cA\sqrt{\Delta}$ be the order in $k(\sqrt{\Delta})$, and let $C(\Delta)$ and $C^+(\Delta)$ denote its wide and narrow class groups. We write $h(\Delta):=|C(\Delta)|$ and $h^+(\Delta):=|C^+(\Delta)|$. 
The quadratic function field $k(\sqrt{\Delta})$ is said to be real if $\deg{\Delta}$ is even and $\chisgn(\Delta)=1$, and imaginary otherwise. 
For $M\in\cA\setminus\{0\}$ and $N\in\cA$, define
$$J_M(N):=\left[\frac{N}{M}\right],$$
where $\left[\frac{\cdot}{\cdot}\right]$ is the Jacobi symbol.
We call $J_M$ the {\it Jacobi character} modulo $M$.\footnote{We use $\left[\frac{\cdot}{M}\right]$ rather than $\left[\frac{M}{\cdot}\right]$, since the former depends only on the residue class modulo $M$, whereas the latter need not.} 
We put
$$M^{(1)}:=\sgn(M)^{-1}M\in\cA^{(1)}$$
and
$$L_{M}(s):=\sum_{F\in\cA^{(1)}}\bigg[\frac{M}{F}\bigg] \|F\|^{-s}.$$
We write
\begin{equation}\label{L-JD-quad-L}
L(s,J_M):=\sum_{F\in\cA^{(1)}}J_M(F)\|F\|^{-s}=L_{(-1)^{\deg M^{(1)}} M^{(1)}}(s),
\end{equation}
where the last equality follows from the quadratic reciprocity law.

Let $D_1,D_2\in\cA$ be two square-free polynomials such that $D_1D_2\notin\cA^2$. For $F_0\in\cA^{(1)}$, let 
$$\Delta=D_1D_2F_0^2\in\cA\setminus\cA^2$$
and $\kappa_\Delta\in\{0,1\}$ such that
\begin{equation}\label{kappa-delta-intro}
(-1)^{\kappa_\Delta}=\chisgn(\Delta)\chisgn(-1)^{\deg\Delta}\in\{\pm1\}.
\end{equation}
We note that $\kappa_\Delta=0$ if $\Delta$ is monic of even degree. 
For a prime $P$, put
\begin{equation}\label{alpha-beta-intro}
\alpha_P:=J_{D_1}(P),\qquad
\beta_P:= (-1)^{\kappa_\Delta \deg P}J_{D_2}(P).
\end{equation}
For every prime $P\nmid F_0$, the quadratic reciprocity law gives
\begin{equation}\label{splitting-alpha-beta}
\left[\frac{\Delta}{P}\right]=\alpha_P\beta_P.
\end{equation}

Let $\chi_{D_1,D_2}^{(\Delta)}:C^+(\Delta)\to\{\pm1\}$ be the {\it genus character}, which is defined in terms of primitive binary quadratic forms and their correspondence with narrow ideal classes in Appendix \ref{app:genus-characters}.
Then we define the {\it genus character $L$-function} for $\mathrm{Re}(s)>1$ by 
\begin{equation}\label{genus-char-L-def}
L(s,\chi_{D_1,D_2}^{(\Delta)}):=\sum_{\cO_\Delta\text{-invertible ideal }\fa\subset\cO_\Delta} 
\frac{\chi_{D_1,D_2}^{(\Delta)}(\fa)}{\|\fN_\Delta(\fa)\|^s},
\end{equation}
where both an $\cO_\Delta$-invertible ideal of $\cO_\Delta$ and the ideal norm
$\fN_\Delta$ are defined in Section \ref{preliminary-sec}. 

Our main results are as follows.
The genus character $L$-function can be represented by the product of  two quadratic $L$-functions and certain local factors related to $F_0$.
\begin{thm}\label{KM-thm-1-Rit-12.3.2}\label{thm-f-eq}
Let $\Delta=D_1D_2F_0^2$ be as above. For any prime $P\in\cA$, let 
$m_P$ be the integer such that $P^{m_P}\parallel F_0$, and let $\alpha_P,\beta_P$ be as in \eqref{alpha-beta-intro}. Let $\kappa_\Delta$ be as in \eqref{kappa-delta-intro} and $\tau={\pi i}/{\log q}$. Define
\begin{equation*}
\epsilon_P(s,\chi_{D_1,D_2}^{(\Delta)}) 
=\frac{\big(1-\alpha_P\|P\|^{-s}\big) \big(1-\beta_P\|P\|^{-s}\big)-\|P\|^{m_P-1-2m_Ps}
\big(\|P\|^{1-s}-\alpha_P\big)\big(\|P\|^{1-s}-\beta_P\big)}{1-\| P \|^{1-2s}}.
\end{equation*}
Then the genus character $L$-function satisfies
\begin{equation}\label{KM-thm-1-Rit-12.3.2-eq}
L(s,\chi_{D_1,D_2}^{(\Delta)})
=L(s,J_{D_1})L(s+\kappa_\Delta\tau,J_{D_2})
\prod_{\substack{P\mid F_0\\P\mathrm{:prime}}} 
\epsilon_P(s,\chi_{D_1,D_2}^{(\Delta)}).
\end{equation}
Furthermore, put 
$$\Lambda(s,\chi_{D_1,D_2}^{(\Delta)})
:=\|F_0\|^{s} G_{D_1}(s)G_{D_2}(s+\kappa_\Delta \tau)
L(s,\chi_{D_1,D_2}^{(\Delta)}),$$
where 
$$G_D(s):=\left\{
\begin{array}{ll}
q^{\frac{\deg{D}-1}{2}s} &\text{if $\deg(D)$ is odd},\\
q^{\frac{\deg{D}-2}{2}s}(1-q^{-s})^{-1} &\text{if $\deg(D)$ is even}.
\end{array}\right.$$
Then $L(s,\chi_{D_1,D_2}^{(\Delta)})$ admits meromorphic continuation to the whole complex plane and satisfies the functional equation
$$\Lambda(s,\chi_{D_1,D_2}^{(\Delta)})=\Lambda(1-s,\chi_{D_1,D_2}^{(\Delta)}).$$
\end{thm}
In contrast to the number field formula in Theorem \ref{KM-L-function-formula}, the product formula \eqref{KM-thm-1-Rit-12.3.2-eq} has a shift of the variable in one of the two quadratic $L$-functions when $\kappa_\Delta=1$.

We further assume that $D_1,D_2$ are monic and satisfy 
$\deg D_1\equiv\deg D_2\pmod 2$. 
Hence $\Delta=D_1D_2F_0^2$ is monic of even degree. For $W\in\XX_\Delta$, 
we interpret $\chi_{D_1,D_2}^{(\Delta)}(W)=\chi_{D_1,D_2}^{(\Delta)}([W]_{\sim_+})$ 
via the correspondence between real quadratic irrational functions
and narrow ideal classes (cf. Appendices \ref{app:continued-fractions} and \ref{app:genus-characters}). 
The following theorem is a function field analogue of Theorem \ref{KMZ-formula}.
\begin{thm}\label{KM-thm-3-Rit-15.1.2}   
Let $\Delta=D_1D_2F_0^2$ be as above. Let 
$$\XX_\Delta^0/\sim=\{[W^{(1)}]_{\sim},\cdots,[W^{(h(\Delta))}]_{\sim}\},$$ 
where each $W^{(j)}$ is a reduced real quadratic irrational function of discriminant $\Delta$. 
Denote by $[M^{(j)}_1,M^{(j)}_2,M^{(j)}_3,\cdots]$ the standard continued fraction of $W^{(j)}$, and by $\partial^{(j)}$ its quasi-period. Put $\tau=\pi i /\log{q}$. 
Then we have: 
\begin{enumerate}[(1)]
\item\label{thm1.4-1} If $D_1\neq1$, $D_2\neq1$ and $\deg{D_1}\equiv\deg{D_2}\equiv0\pmod{2}$, then 
\begin{align*}
&h(gD_1)h(gD_2)\prod_{\substack{P\mid F_0\\P\mathrm{:prime}}}
\epsilon_P\left(\tau,\chi_{D_1,D_2}^{(\Delta)}\right)
\nonumber\allowdisplaybreaks\\
=&\frac{1}{(q+1)^2}(-1)^{\frac{\deg{\Delta}}{2}}
\sum_{j=1}^{h(\Delta)} \chi_{D_1,D_2}^{(\Delta)}(W^{(j)})
\left(4q \sum_{r=1}^{\partial^{(j)}} (-1)^{\deg{M_r^{(j)}}} \frac{\| M_r^{(j)}\|-1}{q-1} + 2t^{(j)}(q-1)\right),
\end{align*}
where $t^{(j)}$ is the number of indices $r$ $(1\le r\le \partial^{(j)})$ for which $\deg{M_r^{(j)}}$ 
is odd.\\
If $\{D_1,D_2\}=\{D,1\}$, $\deg{D}>0$ and $\deg{D}\equiv0\pmod{2}$, then 
\begin{align*}
&h(gD)\prod_{\substack{P\mid F_0\\P\mathrm{:prime}}}
\epsilon_P\left(\tau,\chi_{D,1}^{(\Delta)}\right)
\nonumber\allowdisplaybreaks\\
=&\frac{1}{q+1}(-1)^{\frac{\deg{\Delta}}{2}}
\sum_{j=1}^{h(\Delta)} 
\left(4q
\sum_{r=1}^{\partial^{(j)}} (-1)^{\deg{M_r^{(j)}}} \frac{\| M_r^{(j)}\|-1}{q-1} + 2t^{(j)}(q-1)\right).
\end{align*}
\item\label{thm1.4-2} If $\deg{D_1}\equiv\deg{D_2}\equiv1\pmod{2}$, then $h^+(\Delta)=2h(\Delta)$,
\begin{align*}
&h(D_1)h(D_2)\prod_{\substack{P\mid F_0\\P\mathrm{:prime}}} 
\epsilon_P\left(\tfrac{1-\chisgn(-1)}{2}\tau,\chi_{D_1,D_2}^{(\Delta)}\right)
\nonumber\allowdisplaybreaks\\
=&\chisgn(-1)^{\frac{\deg{\Delta}}{2}}
\sum_{j=1}^{h(\Delta)} \chi_{D_1,D_2}^{(\Delta)}(W^{(j)}) 
 \left(\sum_{r=1}^{\partial^{(j)}} \chisgn{\left((-1)^{r-1}M^{(j)}_r\right)}
\chisgn(-1)^{\deg{M_r^{(j)}}} \frac{\| M_r^{(j)}\|-1}{q-1}\right), 
\end{align*}
and
\begin{align*}
&h(gD_1)h(gD_2)\prod_{\substack{P\mid F_0\\P\mathrm{:prime}}} 
\epsilon_P\left(\tfrac{1+\chisgn(-1)}{2}\tau,\chi_{D_1,D_2}^{(\Delta)}\right)
\nonumber\allowdisplaybreaks\\
=&(-\chisgn(-1))^{\frac{\deg{\Delta}}{2}}
\sum_{j=1}^{h(\Delta)} \chi_{D_1,D_2}^{(\Delta)}(W^{(j)}) 
 \left(  \sum_{r=1}^{\partial^{(j)}} \chisgn{\left((-1)^{r-1}M_r^{(j)} \right)}
(-\chisgn(-1))^{\deg{M_r^{(j)}}} \frac{\| M_r^{(j)}\|-1}{q-1}\right).
\end{align*}
\end{enumerate}
\end{thm}

The following example illustrates Theorem \ref{KM-thm-3-Rit-15.1.2} \eqref{thm1.4-2} when $D_1$ and $D_2$ are not relatively prime, and $F_0\neq1$.
\begin{exmp}\label{example:q3-check}
Let $q=3$ and $g=2\in\FF_3^\times$. Then $\chisgn(-1)=-1$. Put
$$D_1=TQ_1=T^3+T^2+2T,\qquad D_2=TQ_2=T^3+2T^2+2T,\qquad F_0=T+1,$$
where $Q_1=T^2+T+2$ and $Q_2=T^2+2T+2$ are monic irreducible polynomials over
$\FF_3$. Then
$$\Delta=D_1D_2F_0^2=(T^4+1) (T(T+1))^2
=T^8+2T^7+T^6+T^4+2T^3+T^2,$$
and $\lfloor\sqrt{\Delta}\rfloor=T^4+T^3+2$.
A direct sum over the monic polynomials of degrees at most $2$ gives
$$L(s,J_{D_1})=1-2\cdot 3^{-s}+3^{1-2s},\qquad L(s,J_{D_2})=1+2\cdot 3^{-s}+3^{1-2s}$$
(cf. Section \ref{preliminary-sec}).
Put $\tau=\pi i/\log 3$ and then $3^{-\tau}=-1$.
The proof of Lemma \ref{class-number-special-values} gives
$$h(D_1)=6,\qquad h(D_2)=2,\qquad h(gD_1)=2,\qquad h(gD_2)=6.$$
For $P=T+1$, we have
$$m_P=1,\qquad \|P\|=3,\qquad \alpha_P=J_{D_1}(P)=-1,\qquad \beta_P=J_{D_2}(P)=1.$$
With $x=3^{-s}$, the local factor is
$$\epsilon_{T+1}\left(s,\chi_{D_1,D_2}^{(\Delta)}\right)
=\frac{(1+x)(1-x)-x^2(3x+1)(3x-1)}{1-3x^2}
=\frac{1-9x^4}{1-3x^2}=1+3x^2.$$
Thus
$\epsilon_{T+1}\left(0,\chi_{D_1,D_2}^{(\Delta)}\right)
=\epsilon_{T+1}\left(\tau,\chi_{D_1,D_2}^{(\Delta)}\right)=4$,
and both left-hand sides of Theorem \ref{KM-thm-3-Rit-15.1.2} \eqref{thm1.4-2} are 
$$6\cdot2\cdot4=48,\qquad 2\cdot6\cdot4=48.$$ 

Applying \eqref{finite-reduced} gives exactly $42$ reduced elements, which
form four wide classes. We take
$$\begin{array}{lll}
\displaystyle W^{(1)}=\frac{T^4+T^3+2T^2+T+\sqrt{\Delta}}{2T^3},
&&\displaystyle W^{(2)}=\frac{T^4+T^3+T+\sqrt{\Delta}}{2(T^2+2T+2)},\\[4mm]
\displaystyle W^{(3)}=\frac{T^4+T^3+\sqrt{\Delta}}{2(T^2+2T+1)},
&&\displaystyle W^{(4)}=\frac{T^4+T^3+2+\sqrt{\Delta}}{2}.
\end{array}$$
Their partial quotients are
\begin{align*}
(M^{(1)}_1,\ldots,M^{(1)}_7)
&=(T+1,T+1,T^2+1,T+1,T+1,2T+2,T+1),\\
(M^{(2)}_1,\ldots,M^{(2)}_5)
&=(T^2+2T,T+1,2T^2+2T+1,T^2+T+2,2T+2),\\
(M^{(3)}_1,\ldots,M^{(3)}_5)
&=(T^2+2T+1,2T+1,T+2,2T^2+T+2,2T^2+2T),\\
(M^{(4)}_1,\ldots,M^{(4)}_4)
&=(T^4+T^3+2,T,T^2+2T+1,T),
\end{align*}
and their complete quotients satisfy
$$W^{(1)}_8=2W^{(1)}_1,\qquad
W^{(2)}_6=2W^{(2)}_1,\qquad
W^{(3)}_6=2W^{(3)}_1,\qquad
W^{(4)}_5=W^{(4)}_1.$$
Hence $(\partial^{(1)},\partial^{(2)},\partial^{(3)},\partial^{(4)})=(7,5,5,4)$,
and \eqref{local-genus-form} gives
$$\left(\chi_{D_1,D_2}^{(\Delta)}(W^{(j)})\right)_{j=1}^{4}=(1,1,1,1).$$
For $1\leq j\leq4$, put
\begin{align*}
\Sigma_j
&:=\sum_{r=1}^{\partial^{(j)}}\chisgn(-1)^{r-1} \chisgn\left(M^{(j)}_r\right)
\chisgn(-1)^{\deg M^{(j)}_r}\frac{\|M^{(j)}_r\|-1}{q-1},\\
\widetilde{\Sigma}_j
&:=\sum_{r=1}^{\partial^{(j)}}\chisgn(-1)^{r-1} \chisgn\left(M^{(j)}_r\right)
\left(-\chisgn(-1)\right)^{\deg M^{(j)}_r}\frac{\|M^{(j)}_r\|-1}{q-1}.
\end{align*}
A direct substitution of the partial quotients gives
$$(\Sigma_1,\Sigma_2,\Sigma_3,\Sigma_4)=(2,-2,2,46),\qquad
(\widetilde\Sigma_1,\widetilde\Sigma_2,\widetilde\Sigma_3,\widetilde\Sigma_4)=(6,-6,6,42).$$
Therefore, both right-hand sides in Theorem \ref{KM-thm-3-Rit-15.1.2} \eqref{thm1.4-2} are 
$$2-2+2+46=48,\qquad 6-6+6+42=48.$$
\end{exmp}

The remainder of this paper is organized as follows. 
Section \ref{preliminary-sec} recalls class groups of quadratic orders, Jacobi symbols, and quadratic $L$-functions. In Section \ref{pf-main-thm-1-a}, we compute the local Euler factors of the genus character $L$-function and prove Theorem \ref{thm-f-eq}. Sections \ref{Hayes-formula} and \ref{sec-last-proofs} evaluate the relevant partial zeta functions using Hayes' sector formulas and translate these evaluations into continued-fraction data, thereby proving Theorem \ref{KM-thm-3-Rit-15.1.2}.
Appendix \ref{app:quadratic-background} collects the required facts and proofs on 
primitive quadratic forms and real quadratic irrational functions, including their correspondences with class groups, while Appendix \ref{app:genus-characters} gives the form-theoretic construction of general genus characters.

\section{Preliminaries}\label{preliminary-sec}
In this section, we recall basic facts on class groups of quadratic orders, Jacobi symbols and quadratic $L$-functions.

\subsection{Quadratic orders and ideal class groups}
For $\Delta\in\cA\setminus \cA^2$, let $K=k(\sqrt{\Delta})$, and write $\Delta=D_KF^2$, 
where $D_K\in\cA$ is square-free and $F\in\cA^{(1)}$. 
For $Z=A+B\sqrt{\Delta}\in K$ with $A,B\in k$, we define the {\it conjugate} of $Z$ by 
$\bar{Z}=A-B\sqrt{\Delta}$ and the {\it norm} of $Z$ by $\cN(Z)=Z\bar{Z}\in k$. 
The $\cA$-module $\cO_\Delta=\cA+\cA\sqrt{\Delta}$ is called an {\it order} in $K$ of discriminant $\Delta$. 
We denote by $\cO_\Delta^\times$ the unit group of $\cO_\Delta$. 
A {\it fractional $\cO_\Delta$-ideal} $\fa$ is an $\cO_\Delta$-submodule $\fa$ of $K$ for which there exists $Z\in\cO_\Delta\setminus\{0\}$ such that $Z\fa\subset\cO_\Delta$.
A fractional $\cO_\Delta$-ideal $\fa$ is {\it $\cO_\Delta$-invertible} if there exists a fractional $\cO_\Delta$-ideal $\fa_1$ such that $\fa\fa_1=\cO_\Delta$. 
Every nonzero ideal $\fa$ of $\cO_\Delta$ is a free $\cA$-submodule of rank $2$, that is, it has an $\cA$-basis $(w_1,w_2)\in\cO_\Delta^2$ such that $\fa=[w_1,w_2]:=\cA w_1+\cA w_2$. 
A basis $(w_1,w_2)\in K^2$ is called {\it oriented} if $\chisgn((\bar{w}_1w_2-w_1\bar{w}_2)/\sqrt{\Delta})=1$. 
For $\fa=[w_1,w_2]\subset\cO_\Delta$, we define the {\it ideal norm} by 
\begin{equation}\label{Rit-8.5.1}
\fN_\Delta(\fa):=\left(\frac{\bar{w}_1w_2-w_1\bar{w}_2}{\sqrt{\Delta}}\right)^{(1)}
=\sgn{\left(\frac{\bar{w}_1w_2-w_1\bar{w}_2}{\sqrt{\Delta}}\right)}^{-1} 
\left(\frac{\bar{w}_1w_2-w_1\bar{w}_2}{\sqrt{\Delta}}\right).
\end{equation}

A nonzero $\cA$-submodule $\fa\subset \cO_\Delta$ is an ideal of $\cO_\Delta$ 
if and only if it admits an $\cA$-basis of the form 
\begin{equation}\label{basis-of-the-form}
\fa=[2SA,S(B+\sqrt{\Delta})],\qquad \Delta=B^2-4AC
\end{equation}
for some $A,B,C,S\in \cA$.
Such an ideal is $\cO_\Delta$-invertible if and only if $(A,B,C)=1$.
It is called {\it $\cO_\Delta$-primitive} if one may take $S=1$, and 
{\it $\cO_\Delta$-regular} if it is both $\cO_\Delta$-invertible and $\cO_\Delta$-primitive.

Let $I(\Delta)$ be the group of {\it $\cO_\Delta$-invertible fractional $\cO_\Delta$-ideals}. 
Let $P(\Delta)$ (resp., $P^+(\Delta)$) be the group of nonzero principal ideals in the wide (resp., narrow) sense, that is, 
$$P(\Delta):=\{Z\cO_\Delta\,\mid\, Z\in K^\times\} \qquad(\text{resp., }
P^+(\Delta):=\{Z\cO_\Delta\,\mid\, Z\in K^\times,\,\chisgn(\cN(Z))=1\}).$$ 
The quotient group $C(\Delta):=I(\Delta)/P(\Delta)$ (resp., $C^+(\Delta):=I(\Delta)/P^+(\Delta)$) is called the {\it wide}
(resp., {\it narrow}) {\it ideal class group} of $\cO_\Delta$. We call $h(\Delta):=|C(\Delta)|$ (resp., $h^+(\Delta):=|C^+(\Delta)|$) the {\it wide} (resp., {\it narrow}) {\it class number} of $\cO_\Delta$.
The natural map $C^+(\Delta)\to C(\Delta)$ is surjective with kernel $P(\Delta)/P^+(\Delta)$, whose order is at most 2. Hence $h^+(\Delta)=h(\Delta)$ or $h^+(\Delta)=2h(\Delta)$.

\subsection{Jacobi symbols and quadratic $L$-functions}
We introduce the Jacobi symbol and the quadratic $L$-function.
For $M\in\cA$ and $N\in\cA\setminus\{0\}$, the {\it Jacobi symbol} is defined by
$$\left[\frac{M}{N}\right]:=\prod_{i=1}^{r}\left[\frac{M}{P_i}\right]^{e_i},$$
where $P_i\in \cA$ are primes,  $N=\zeta P_1^{e_1}\cdots P_r^{e_r}$ with 
$\zeta\in \FF_q^\times$, and for any prime $P\in\cA$
$$\left[\frac{M}{P}\right]=\left\{
\begin{array}{rl}
1 &\text{if }P\nmid M \text{ and }Z^2\equiv M \,\,(\text{mod }P) \text{ is solvable in }\cA,\\
-1 &\text{if }P\nmid M \text{ and }Z^2\equiv M \,\,(\text{mod }P) \text{ is not solvable in }\cA,\\
0 &\text{if }P\mid M.
\end{array}\right.$$
The Jacobi symbol has the following properties (cf. \cite[Chapter 3 and Proposition 4.2]{Ros02}). 
\begin{prop}\label{Rit-9.1.3}  
Let $\zeta\in\FF_q^\times$, $M_1,M_2\in\cA$, and $M,N\in\cA\setminus\{0\}$.
\begin{enumerate}[(1)]
\item\label{Rit-9.1.3-1} $\big[\frac{M_1}{N}\big]\big[\frac{M_2}{N}\big]=\big[\frac{M_1M_2}{N}\big]$.
\item\label{Rit-9.1.3-2} $M_1\equiv M_2 \pmod{N} \quad\Rightarrow\quad \big[\frac{M_1}{N}\big]=\big[\frac{M_2}{N}\big]$. 
\item\label{Rit-9.1.3-3} (law of quadratic reciprocity) \quad 
$\left[\frac{M}{N}\right]=\left[\frac{N}{M}\right]\chisgn(-1)^{\deg{M}\cdot\deg{N}}
\chisgn(M)^{\deg{N}}\chisgn(N)^{\deg{M}}$.
\item\label{Rit-9.1.3-4} (supplementary law) \quad 
$\left[\frac{\zeta}{N}\right]=\chisgn(\zeta)^{\deg{N}}$. 
\item\label{Rit-9.1.3-5} 
If $N$ is not a constant multiple of a square in $\cA$ and $M$ runs through a residue system mod $N$, then
$\sum_{M\,\,(\text{mod }N)}\left[\frac{M}{N}\right]=0$. 
\end{enumerate}
\end{prop}

For nonzero $M\in\cA$, define 
$$L_{M}(s):=\sum_{F\in\cA^{(1)}}\bigg[\frac{M}{F}\bigg] \|F\|^{-s}
=\prod_{P\mathrm{:prime}} \bigg(1-\bigg[\frac{M}{P}\bigg] \|P\|^{-s}\bigg)^{-1},\qquad 
\mathrm{Re}(s)>1.$$
We write $\zeta_{\cA}(s)=L_1(s)=(1-q^{1-s})^{-1}$.
By the supplementary law, we have
\begin{equation}\label{D-gD-input}
L_{gM}(s)=L_{M}\bigg(s+\frac{\pi i}{\log{q}}\bigg)
\end{equation}
and 
$$L_g(s)=\zeta_{\cA}\bigg(s+\frac{\pi i}{\log{q}}\bigg)=\frac{1}{1+q^{1-s}},\qquad 
\mathrm{Re}(s)>1.$$
Assume that $D\in\cA$ is square-free and $\deg(D)\ge1$. Then $L_D(s)$ is called the {\it quadratic $L$-function} associated with $k(\sqrt{D})$ and is a polynomial in $q^{-s}$ of degree at most $\deg(D)-1$
(cf. \cite[Proposition 4.3]{Ros02}).
We define 
$$\Lambda_D(s):=\left\{
\begin{array}{ll}
q^{\frac{\deg{D}-1}{2}s}L_D(s) &\text{if $\deg(D)$ is odd},\\
q^{\frac{\deg{D}-2}{2}s}(1-q^{-s})^{-1}L_D(s) &\text{if $\deg(D)$ is even and $\chi_{\sgn}(D)=1$},\\
q^{\frac{\deg{D}-2}{2}s}(1+q^{-s})^{-1}L_D(s) &\text{if $\deg(D)$ is even and $\chi_{\sgn}(D)=-1$}.
\end{array}\right.$$
Then we have the functional equation
\begin{equation}\label{quad-L-fcn-eq}
\Lambda_D(s)=\Lambda_D(1-s)
\end{equation}
(cf. \cite[Theorem 5.9 and the proof of Theorem 17.8A]{Ros02}).
Combining \eqref{quad-L-fcn-eq} with \cite[Theorem 17.8A]{Ros02}, we obtain
\begin{equation}\label{Rit-9.3.1-new}
L_D(0)=\left\{
\begin{array}{ll}
h(D) &\text{if $k(\sqrt{D})$ is imaginary},\\
0 &\text{if $k(\sqrt{D})$ is real}.
\end{array}\right.
\end{equation}
The preceding definition of $\Lambda_D(s)$ and the functional equation \eqref{quad-L-fcn-eq} remain valid also when $\deg{D}=0$.

\section{Proof of Theorem \ref{thm-f-eq}}\label{pf-main-thm-1-a}
We continue to use the notation of Theorem \ref{thm-f-eq} 
and write $\Delta=D_KF^2$, where $D_K\in\cA$ is square-free and $F\in\cA^{(1)}$ is the conductor of $\cO_\Delta$.
For each prime $P$, $m_P=v_P(F_0)$ and 
$$v_P(F)=\left\{\begin{array}{ll}
v_P(F_0)+1 &\text{if }P\mid D_1 \text{ and }P\mid D_2,\\
v_P(F_0) &\text{otherwise},
\end{array}\right.$$
where $v_P$ denotes the $P$-adic valuation.
The ideal norm $\fN_\Delta$ and the genus character $\chi_{D_1,D_2}^{(\Delta)}$ are multiplicative, and every $\cO_\Delta$-invertible ideal of $\cO_\Delta$ decomposes uniquely into components of $P$-power norm. Hence the genus character $L$-function has the Euler product expression
$$L(s,\chi_{D_1,D_2}^{(\Delta)})
=\prod_{P\text{:prime}}L_P(s,\chi_{D_1,D_2}^{(\Delta)}),$$
where 
\begin{equation}\label{P-factor-def}
L_P(s,\chi_{D_1,D_2}^{(\Delta)}):=\sum_{t\ge0}\sum_{\substack{\cO_\Delta\text{-invertible ideal }\fa\subset\cO_\Delta\\ \fN_\Delta(\fa)=P^t}}
\frac{\chi_{D_1,D_2}^{(\Delta)}(\fa)}{\|\fN_\Delta(\fa)\|^s}
=\sum_{t\ge0}\frac{\eta_1(P^t)}{\|P\|^{ts}},
\end{equation}
\begin{eqnarray*}
\eta_1(P^t)
&:=&\sum_{\substack{\cO_\Delta\text{-invertible ideal }\fa\subset\cO_\Delta\\ \fN_\Delta(\fa)=P^t}} \chi_{D_1,D_2}^{(\Delta)}(\fa)
=\sum_{j=0}^{\lfloor t/2\rfloor}
\sum_{\substack{\cO_\Delta\text{-regular ideal }\fb\subset\cO_\Delta\\ \fN_\Delta(\fb)=P^{t-2j}}}
\chi_{D_1,D_2}^{(\Delta)}(\fb)
=\sum_{j=0}^{\lfloor t/2\rfloor}\theta_1(P^{t-2j})
\end{eqnarray*}
and
$$\theta_1(P^a):=\sum_{\substack{\cO_\Delta\text{-regular ideal }\fb\subset\cO_\Delta\\ \fN_\Delta(\fb)=P^{a}}}
\chi_{D_1,D_2}^{(\Delta)}(\fb).$$
We analogously define $\eta_0(P^t)$ and $\theta_0(P^a)$ by replacing all of the values $\chi_{D_1,D_2}^{(\Delta)}(\fb)$ with $1$ (cf. \cite[Section 2]{KM20}).

The computation of the local factor $L_P(s,\chi_{D_1,D_2}^{(\Delta)})$ is divided into two cases according to whether $P$ divides the conductor $F$. If $P\nmid F$, every
$\cO_\Delta$-invertible ideal of $\cO_\Delta$ appearing in the sum defining
$\eta_1(P^t)$ is prime to the conductor. Hence the unique factorization
theorem for ideals prime to the conductor applies, and all prime
factors lie above $P$. Thus, in Section \ref{subsec-3.1}, we compute
$\eta_1(P^t)$ directly. If $P\mid F$, we instead classify the
$\cO_\Delta$-regular ideals of $\cO_\Delta$ with $P$-power norm and compute
$\theta_i(P^a)$ for $i=0,1$. The values $\eta_i(P^t)$ are then
recovered from these computations in Section \ref{subsec-3.2}.

\subsection{Local factors away from the conductor}\label{subsec-3.1}
\begin{lem}\label{lemma-4.1}
Let $P$ be a prime such that $P\nmid F$. Then one has 
$$L_P(s,\chi_{D_1,D_2}^{(\Delta)})
=\frac{1}{(1-\alpha_P\|P\|^{-s})(1-\beta_P\|P\|^{-s})}.$$
\end{lem}
\begin{proof}
The unique factorization into prime ideals holds for any nonzero ideal $\fa$ of $\cO_\Delta$ prime to $F$, that is, $(\fN_\Delta(\fa),F)=1$. 
Let $t>0$ and $\fa$ an $\cO_\Delta$-invertible ideal of $\cO_\Delta$ 
such that $\fN_\Delta(\fa)=P^t$. Let $\fq$ be any prime ideal of $\cO_\Delta$ 
such that $\fa\subset\fq$. 
Then $\fq\cap\cA=P\cA$ since $P^t\in P^t\cO_\Delta=\fa\bar{\fa}\subset \fq$, where $\bar{\fa}=\{Z\mid \bar{Z}\in\fa\}$. 

Since $P\nmid F$, we have that $P\nmid F_0$ and $\left[\frac{\Delta}{P}\right]=\alpha_P\beta_P$ by \eqref{splitting-alpha-beta}. 
If $\left[\frac{\Delta}{P}\right]=1$, then $\alpha_P\beta_P=1$ 
and $P\cO_\Delta=\fP \bar{\fP}$ with $\fP\neq \bar{\fP}$, where $\fP=[2P,X+\sqrt{\Delta}]$ is a prime ideal of norm $P$ and  $X^2\equiv \Delta \pmod{P}$. Then
$$\chi_{D_1,D_2}^{(\Delta)}(\fP)=\chi_{D_1,D_2}^{(\Delta)}(\bar{\fP})
=\left[\frac{P}{D_1}\right]=\alpha_P=\beta_P$$
by \eqref{Qw-def} and \eqref{signed-genus-reciprocal-value}; so, we get
\begin{equation}\label{eq-case-1-1}
\eta_1(P^t)=\#\{\fP^a\bar{\fP}^b:a+b=t,\,a,b\in\ZZ_{\ge0}\}\times\alpha_P^t
=(t+1)\alpha_P^t.
\end{equation}
If $\left[\frac{\Delta}{P}\right]=-1$, then $\alpha_P\beta_P=-1$ and 
$P\cO_\Delta=P[2,\sqrt{\Delta}]$ is a prime ideal of norm $P^2$. 
Then
$$\chi_{D_1,D_2}^{(\Delta)}(P\cO_\Delta)=\left[\frac{1}{D_1}\right]=1$$
by \eqref{Qw-def} and \eqref{signed-genus-reciprocal-value}; so, we have
\begin{equation}\label{eq-case-1-2}
\eta_1(P^t)=\left\{\begin{array}{ll}
1 &\text{for }2\mid t,\\ 0 &\text{for }2\nmid t.
\end{array}\right.
\end{equation}
If $\left[\frac{\Delta}{P}\right]=0$, then $P\mid D_1D_2$, $P^2\nmid D_1D_2$ and $P\cO_\Delta=\fP^2$, where $\fP=[2P,\sqrt{\Delta}]$.
Then 
$$\chi_{D_1,D_2}^{(\Delta)}(\fP)=\left\{\begin{array}{ll}
\alpha_P &\text{for }P\nmid D_1,\\ \beta_P &\text{for }P\nmid D_2
\end{array}\right.$$
by \eqref{Qw-def}, \eqref{signed-genus-reciprocal-value} and \eqref{general-symmetry-obstruction}; so, we obtain 
\begin{equation}\label{eq-case-1-3}
\eta_1(P^t)=\left\{\begin{array}{ll}
\alpha_P^t &\text{for }P\nmid D_1,\\ \beta_P^t &\text{for }P\nmid D_2.
\end{array}\right.
\end{equation}
We see that \eqref{eq-case-1-1}, \eqref{eq-case-1-2} and \eqref{eq-case-1-3} hold for $t=0$.
By \eqref{P-factor-def}, \eqref{eq-case-1-1}, \eqref{eq-case-1-2}, \eqref{eq-case-1-3} and the following identity 
$$\frac{1}{(1-\alpha_P\|P\|^{-s})(1-\beta_P\|P\|^{-s})}=\sum_{t\ge0}\sum_{\substack{a+b=t,\\a,b\in\ZZ_{\ge0}}}
\frac{\alpha_P^a \beta_P^b}{\|P\|^{ts}},$$
we get
\begin{equation*}
L_P(s,\chi_{D_1,D_2}^{(\Delta)})
=\sum_{t\ge0}\frac{\eta_1(P^t)}{\|P\|^{ts}}
=\frac{1}{(1-\alpha_P\|P\|^{-s})(1-\beta_P\|P\|^{-s})}.\qedhere
\end{equation*}
\end{proof}

\smallskip

\subsection{Local factors at primes dividing the conductor}\label{subsec-3.2}
The following lemma gives the list of all $\cO_\Delta$-regular ideals of $\cO_\Delta$ whose norm is any power of $P\mid F$, which is a function field analogue of \cite[p. 263]{BP72}.
\begin{lem}\label{pl-invertible} 
Let $\Delta=D_KF^2$ with $D_K\in\cA$ the discriminant of $K=k(\sqrt{\Delta})$ and $F\in\cA^{(1)}$ the conductor of $\cO_\Delta$. 
Let $P$ be any prime dividing $F$. 
Let $n$ be the largest integer such that $P^{2n}\mid \Delta$ and $E:=\Delta/P^{2n}$. 
Then there are four types of $\cO_\Delta$-regular ideals $[2P^l,Z+\sqrt{\Delta}]$ with $l>0:$ 
\begin{enumerate}[(1)]
\item $l=2m<2n$, $Z=P^mM$, $(M,P)=1$, 
\item $l=2n$, $Z=P^nM$, $(M^2-E,P)=1$, 
\item $l=2n+1$, $P\mid E$, $Z=P^nM$, $P\mid M$, 
\item $l>2n$, $\big[\frac{E}{P}\big]=1$, $Z=P^nM$, $\pm M=M_0+SP^{l-2n}$ with $(S,P)=1$. 
\end{enumerate}
Here, $M_0$ denotes a fixed solution of $M_0^2\equiv E \pmod{P^{l-2n+1}}$.
\end{lem}
\begin{proof}
Since $P\mid F$ and $\Delta=Z^2-4P^lC$ with $C\in\cA$, 
one has $P^2\mid \Delta$ and $P\mid Z$.
Therefore, $[2P^l,Z+\sqrt{\Delta}]$ is $\cO_\Delta$-regular if and only if 
$C=(Z^2-\Delta)/(4P^l)\in\cA$ and $(P^l,Z,C)=1$, that is, 
\begin{equation}\label{vp-condition}
v_P(Z^2-\Delta)=l.
\end{equation}
Let $v_P(\Delta)=2n+\delta$ with $\delta\in\{0,1\}$.
If $l<2n$, \eqref{vp-condition} forces $l=2v_P(Z)$, giving the case (1).
If $l=2n$, writing $Z=P^nM$ gives the case (2). 
If $l>2n$, \eqref{vp-condition} becomes 
$$v_P(M^2-E)=l-2n.$$ 
When $P\mid E$, this is possible only for $l=2n+1$, giving the case (3). 
When $P\nmid E$, it is possible only for $\big[\frac{E}{P}\big]=1$; so, Hensel lifting gives the case (4).
\end{proof}

The following character sum is a function field analogue of \cite[Proposition 4.1]{Sma91}.
\begin{lem}\label{character-sum}
Let $P\in\cA$ be a prime. Let $f(X)=AX^2+BX+C\in \cA[X]$ be a quadratic polynomial of discriminant $\Delta\in\cA$ with $P\nmid A$. Then 
$$\sum_{X\,(\mathrm{mod}\,P)}J_P(f(X))
=\left\{\begin{array}{ll}
-J_P(A) &\text{if }P\nmid \Delta,\\
(\|P\|-1)J_P(A) &\text{if }P\mid \Delta.
\end{array}\right.$$
\end{lem}
\begin{proof}
Using $4Af(X)=(2AX+B)^2-\Delta$ and the fact that $2A$ is invertible modulo $P$, we obtain
$$\sum_{X\,(\mathrm{mod}\,P)}J_P(f(X))
=J_P(A)\sum_{X\,(\mathrm{mod}\,P)}J_P(X^2-\Delta).$$
If $P\mid\Delta$, the last sum is $\|P\|-1$.  
Suppose that $P\nmid\Delta$. 
Since $1+J_P(Z)$ is the number of solutions modulo $P$ of $Y^2\equiv Z\pmod P$, 
we have
\begin{align*}
\sum_{X\,(\mathrm{mod}\,P)}J_P(X^2-\Delta)
&=-\|P\| + \#\{(X,Y) :  X^2-Y^2\equiv\Delta\pmod P\} \allowdisplaybreaks\\
&=-\|P\| + \#\{(U,V):UV\equiv\Delta\pmod P\} \allowdisplaybreaks\\
&=-1. \qedhere
\end{align*}
\end{proof}

\begin{lem}\label{theta-0-1} 
Let $P$ be a prime dividing $F$, let $n$ be the largest integer such that $P^{2n}\mid \Delta$ and put $E:=\Delta/P^{2n}$. 
The values $\theta_i(P^a)$ with $a>0$ and $i\in\{0,1\}$ are given as follows. 
\begin{enumerate}[(1)]
\item\label{case-4.4.1} For $a<2n$, 
\begin{enumerate}[(a)]
\item\label{case-4.4.1-a} if $a$ is odd, then $\theta_0(P^a)=0$ and $\theta_1(P^a)=0$, 
\item\label{case-4.4.1-b} if $a=2m>0$ is even, then $\theta_0(P^a)=\|P\|^m-\|P\|^{m-1}$ and 
$\theta_1(P^a)=\|P\|^m-\|P\|^{m-1}$. 
\end{enumerate}
\item\label{case-4.4.2} For $a\ge2n$, we make a case distinction according to the value of $\big[\frac{E}{P}\big]$. 
\begin{enumerate}[{Case} 1.]
\item\label{case-4.4.2-1} Suppose $\big[\frac{E}{P}\big]=1$.
\begin{enumerate}[(a)]
\item\label{case-4.4.2-1-a} If $a>2n$, then $\theta_0(P^a)=2(\|P\|^n-\|P\|^{n-1})$ and 
$$\theta_1(P^a)=\left\{\begin{array}{ll}
\alpha_P^a\cdot 2(\|P\|^n-\|P\|^{n-1}) &\text{if }P\nmid D_1,\\
0 &\text{if }P\mid D_1.
\end{array}\right.$$
\item\label{case-4.4.2-1-b} If $a=2n$, then $\theta_0(P^{2n})=\|P\|^{n-1}(\|P\|-2)$ and 
$$\theta_1(P^{2n})=\left\{\begin{array}{ll}
\|P\|^{n-1}(\|P\|-2) &\text{if }P\nmid D_1,\\
-\|P\|^{n-1} &\text{if }P\mid D_1.
\end{array}\right.$$
\end{enumerate}
\item\label{case-4.4.2-2} Suppose $\big[\frac{E}{P}\big]=-1$.
\begin{enumerate}[(a)]
\item\label{case-4.4.2-2-a} If $a>2n$, then $\theta_0(P^a)=0$ and $\theta_1(P^a)=0$. 
\item\label{case-4.4.2-2-b} If $a=2n$, then $\theta_0(P^{2n})=\|P\|^n$ and 
$$\theta_1(P^{2n})=\left\{\begin{array}{ll}
\|P\|^n &\text{if }P\nmid D_1,\\
-\|P\|^{n-1} &\text{if }P\mid D_1.
\end{array}\right.$$
\end{enumerate}
\item\label{case-4.4.2-3} Suppose $P\mid E$. 
\begin{enumerate}[(a)]
\item\label{case-4.4.2-3-a} If $a>2n+1$, then $\theta_0(P^a)=0$, $\theta_1(P^a)=0$.
\item\label{case-4.4.2-3-b} If $a=2n$, then $\theta_0(P^{2n})=\|P\|^n-\|P\|^{n-1}$, $\theta_1(P^{2n})=\|P\|^n-\|P\|^{n-1}$.
\item\label{case-4.4.2-3-c} If $a=2n+1$, then $\theta_0(P^{2n+1})=\|P\|^n$ and 
$$\theta_1(P^{2n+1})=\left\{\begin{array}{ll}
\alpha_P\|P\|^n &\text{if }P\nmid D_1,\\
\beta_P\|P\|^n &\text{if }P\mid D_1.
\end{array}\right.$$
\end{enumerate}
\end{enumerate}
\end{enumerate}
\end{lem}

\begin{proof}
An $\cO_\Delta$-regular ideal $\fb$ of $\cO_\Delta$ with norm $P^a$
has the form
$$\fb=[2P^a,B+\sqrt{\Delta}], \qquad f_{\fb}=[P^a,B,C], \qquad
C=\frac{B^2-\Delta}{4P^a}, \qquad (P^a,B,C)=1,$$
where $B,C\in\cA$ and $f_{\fb}$ is the corresponding primitive form of
discriminant $\Delta$. Then the value
$\chi_{D_1,D_2}^{(\Delta)}(\mathfrak b)$ is computed as
$$\chi_{D_1,D_2}^{(\Delta)}(f_{\fb})=\prod_{Q\mid D_1} J^{(Q)}(f_{\fb}),$$
where the product is over prime divisors of $D_1$ (cf. Appendix \ref{app:genus-characters}).

{(i) The vanishing cases.}
By Lemma \ref{pl-invertible}, there are no $\mathcal O_\Delta$-regular ideals of norm
$P^a$ when $a<2n$ is odd, when $a>2n$ and $\left[\frac{E}{P}\right]=-1$,
or when $a>2n+1$ and $P\mid E$.
This proves all the asserted vanishing statements: {(1) (a), (2) Case 2(a), 
and (2) Case 3(a)}.

{(ii) The cases $a=2m<2n$, and $a=2n$ with $P\mid E$.}
In the second case, we put $m=n$. 
By Lemma \ref{pl-invertible}, in both cases
$$\fb=[2P^{2m},P^mM+\sqrt{\Delta}],\qquad M\,\, (\bmod{P^m}),\qquad P\nmid M.$$
Hence
$$\theta_0(P^{2m})=\sum_{\substack{M\,(\mathrm{mod}\,P^m)\\P\nmid M}}1
=\|P\|^m-\|P\|^{m-1}.$$
The corresponding primitive form is 
$f_{\fb}=\left[P^{2m},P^mM,(M^2-P^{2(n-m)}E)/4\right]$. 
For every prime $Q\mid D_1$ with $Q\neq P$, we have
$$J^{(Q)}(f_{\fb})=J_Q(P)^{2m}=1.$$
If $P\mid D_1$, then
$$J^{(P)}(f_{\fb})=J_P\left((M^2-P^{2(n-m)}E)/4\right)=J_P(M^2/4)=1.$$
Therefore $\chi_{D_1,D_2}^{(\Delta)}(\mathfrak b)=1$ for every such ideal, and consequently
$$\theta_1(P^{2m})=\theta_0(P^{2m})=\|P\|^m-\|P\|^{m-1},$$
which proves {(1) (b) and (2) Case 3(b)}.

{(iii) The case $a=2n$ with $P\nmid E$.}
By Lemma \ref{pl-invertible},
$$\mathfrak b=[2P^{2n},P^nM+\sqrt{\Delta}],\qquad 
M\,\,(\bmod P^n),\qquad P\nmid M^2-E.$$
Since every residue class modulo $P$ has $\|P\|^{n-1}$ lifts modulo $P^n$, we obtain
$$\theta_0(P^{2n})=\begin{cases}
\|P\|^n-2\|P\|^{n-1}=\|P\|^{n-1}(\|P\|-2) &\text{if } \left[\frac{E}{P}\right]=1,\\
\|P\|^n &\text{if } \left[\frac{E}{P}\right]=-1.
\end{cases}$$
The corresponding primitive form is $f_{\fb}=\left[P^{2n},P^nM,(M^2-E)/4\right]$.
For every prime $Q\mid D_1$ with $Q\neq P$,
$$J^{(Q)}(f_{\fb})=J_Q(P)^{2n}=1.$$
Thus, if $P\nmid D_1$, then $\chi_{D_1,D_2}^{(\Delta)}(\mathfrak b)=1$,
and hence
$$\theta_1(P^{2n})=\theta_0(P^{2n}).$$
Suppose that $P\mid D_1$. Then
$$\chi_{D_1,D_2}^{(\Delta)}(\fb)=J_P\left((M^2-E)/4\right)=J_P(M^2-E)$$
and therefore Lemma \ref{character-sum} gives
$$\theta_1(P^{2n})
=\sum_{\substack{M\,(\mathrm{mod}\,P^n)\\P\nmid M^2-E}}J_P(M^2-E)
=\sum_{M\,(\mathrm{mod}\,P^n)}J_P(M^2-E)
=-\|P\|^{n-1}.$$
This proves {(2) Case 1(b) and (2) Case 2(b)}.

{(iv) The case $a>2n$ with $\left[\frac{E}{P}\right]=1$.}
By Lemma \ref{pl-invertible},
$$\fb=[2P^a,P^nM+\sqrt{\Delta}], \qquad
\pm M=M_0+SP^{a-2n}, \qquad
P\nmid S,$$
where $S$ runs through a residue system modulo $P^n$. 
The two signs give disjoint classes since $P\nmid M_0$. Thus
$$\theta_0(P^a)=2\sum_{\substack{S\,(\mathrm{mod}\,P^n)\\P\nmid S}}1
=2\bigl(\|P\|^n-\|P\|^{n-1}\bigr).$$
The corresponding primitive form is 
$f_{\fb}=\left[P^a,P^nM,(M^2-E)/(4P^{a-2n})\right]$. 
Since $M_0^2\equiv E\pmod{P^{a-2n+1}}$, we have
$$\frac{M^2-E}{4P^{a-2n}} \equiv \frac{M_0S}{2} \pmod{P}.$$
Therefore, for every prime $Q\mid D_1$,
$$J^{(Q)}(f_{\fb})=\begin{cases}
J_Q(P)^a &\text{if }Q\neq P,\\
J_Q(M_0S/2) &\text{if }Q=P.
\end{cases}$$
If $P\nmid D_1$, then
$$\chi_{D_1,D_2}^{(\Delta)}(\fb)=\prod_{Q\mid D_1}J_Q(P)^a
=J_{D_1}(P)^a=\alpha_P^a,$$
and hence
$$\theta_1(P^a)=2\alpha_P^a (\|P\|^n-\|P\|^{n-1}).$$
Suppose that $P\mid D_1$. Then
$$\chi_{D_1,D_2}^{(\Delta)}(\fb)=J_{D_1/P}(P)^a J_P(M_0S/2).$$
It follows 
from Proposition \ref{Rit-9.1.3} \eqref{Rit-9.1.3-5}
that
$$\theta_1(P^a)=2 J_{D_1/P}(P)^a J_P(M_0/2)
\sum_{\substack{S\,(\mathrm{mod}\,P^n)\\ P \nmid S}}J_P(S)=0.$$
This proves {(2) Case 1(a)}.

{(v) The case $a=2n+1$ with $P\mid E$.}
By Lemma \ref{pl-invertible},
$$\fb=[2P^{2n+1},P^nM+\sqrt{\Delta}],\qquad M\,\, (\bmod{P^{n+1}}),\qquad P\mid M.$$
Consequently,
$$\theta_0(P^{2n+1})
=\sum_{\substack{M\,(\mathrm{mod}\,P^{n+1})\\ P \mid M}}1
=\|P\|^n.$$
The corresponding primitive form is 
$f_{\fb}=\left[P^{2n+1},P^nM,(M^2-E)/(4P)\right]$.
If $P\nmid D_1$, then
$$\chi_{D_1,D_2}^{(\Delta)}(f_{\mathfrak b})
=\prod_{Q\mid D_1}J_Q(P)^{2n+1}
=J_{D_1}(P)
=\alpha_P.$$
Hence
$$\theta_1(P^{2n+1})=\alpha_P\|P\|^n.$$
Suppose that $P\mid D_1$. Since $P\mid E$, we have $P\nmid D_2$. 
By \eqref{general-symmetry-obstruction},
$$\begin{aligned}
\chi_{D_1,D_2}^{(\Delta)}(f_{\fb})
&=\chi_{\mathrm{sgn}}(-1)^{\deg\Delta\,(2n+1)\deg P}
\chi_{\mathrm{sgn}}(\Delta)^{(2n+1)\deg P} 
\chi_{D_2,D_1}^{(\Delta)}(f_{\fb})\\
&=\left(\chi_{\mathrm{sgn}}(\Delta)\chi_{\mathrm{sgn}}(-1)^{\deg\Delta}\right)^{\deg P} 
J_{D_2}(P)\\
&=\beta_P.
\end{aligned}$$
Consequently,
$$\theta_1(P^{2n+1})=\beta_P\|P\|^n.$$
This proves {(2) Case 3(c)}.
\end{proof}

\begin{lem}\label{eta-0-1}
With the notation of Lemma \ref{theta-0-1}, 
the values $\eta_i(P^t)$ with $i\in\{0,1\}$ are as follows.
\begin{enumerate}[(1)]
\item For $t<2n$, 
$$\eta_0(P^t)=\eta_1(P^t)=\left\{
\begin{array}{ll}
0 &\text{if }t\text{ is odd},\\
\|P\|^{t/2} &\text{if }t\text{ is even}.
\end{array}\right.$$
\item For $t\ge 2n$, we make a case distinction according to the value of $\big[\frac{E}{P}\big]$. 
\begin{enumerate}[{Case} 1.]
\item Suppose that $\big[\frac{E}{P}\big]=1$. Then
$\eta_0(P^t)=(t-2n+1)(\|P\|^n-\|P\|^{n-1})$ and
$$\eta_1(P^t)=\left\{\begin{array}{ll}
\alpha_P^t(t-2n+1)(\|P\|^n-\|P\|^{n-1}) &\text{if }P\nmid D_1,\\
0 &\text{if }P\mid D_1.
\end{array}\right.$$
\item Suppose that $\big[\frac{E}{P}\big]=-1$. If $t$ is odd, then 
$\eta_0(P^t)=\eta_1(P^t)=0$. 
If $t$ is even, then $\eta_0(P^t)=\|P\|^n+\|P\|^{n-1}$ and
$$\eta_1(P^t)=\left\{\begin{array}{ll}
\|P\|^n+\|P\|^{n-1} &\text{if }P\nmid D_1,\\
0 &\text{if }P\mid D_1.
\end{array}\right.$$
\item Suppose that $P\mid E$. Then 
$\eta_0(P^t)=\|P\|^n$ and
$$\eta_1(P^t)=\left\{\begin{array}{ll}
\alpha_P^t \|P\|^n &\text{if }P\nmid D_1,\\
\beta_P^t \|P\|^n &\text{if }P\mid D_1.
\end{array}\right.$$
\end{enumerate}
\end{enumerate}
\end{lem}
\begin{proof}
Note that $\theta_0(1)=\theta_1(1)=1$. 
It follows from Lemma \ref{theta-0-1} and 
$\eta_i(P^t)=\sum_{j=0}^{\lfloor t/2\rfloor}\theta_i(P^{t-2j})$.  
\end{proof}

\subsection{Proof of Theorem \ref{thm-f-eq}}
\noindent
\begin{proof}[Proof of Theorem \ref{thm-f-eq}]
For $P\mid F$ and $x=\|P\|^{-s}$, we claim that
\begin{equation}\label{local-factor-at-conductor}
L_P(s,\chi_{D_1,D_2}^{(\Delta)})
=\frac{(1-\alpha_P x)(1-\beta_P x)
-\|P\|^{m_P-1}x^{2m_P}(\|P\|x-\alpha_P)(\|P\|x-\beta_P)}
{(1-\alpha_P x)(1-\beta_P x)(1-\|P\|x^2)}.
\end{equation}
Let $n$ be the integer such that $P^n\parallel F$ and let $E:=\Delta/P^{2n}$.
First suppose that $P\nmid D_1D_2$.  Then $n=m_P$ and by the quadratic reciprocity law,
$$\left[\frac{E}{P}\right]=\alpha_P \beta_P.$$
If $\left[\frac{E}{P}\right]=1$, then $\alpha_P=\beta_P\in\{\pm1\}$ and Lemma \ref{eta-0-1} gives
\begin{align*}
L_P(s,\chi_{D_1,D_2}^{(\Delta)})
&=\sum_{r=0}^{n-1}\|P\|^r x^{2r} + (\|P\|^n-\|P\|^{n-1})x^{2n}\sum_{j\geq0}(j+1)(\alpha_P x)^j\\
&=\frac{1-\|P\|^nx^{2n}}{1-\|P\|x^2} + \frac{(\|P\|^n-\|P\|^{n-1})x^{2n}}{(1-\alpha_P x)^2}\\
&=\frac{(1-\alpha_P x)^2-\|P\|^{n-1}x^{2n}(1-\alpha_P \|P\| x)^2}
{(1-\alpha_P x)^2(1-\|P\| x^2)}.
\end{align*}
If $\left[\frac{E}{P}\right]=-1$, then $\alpha_P=-\beta_P\in\{\pm1\}$ and 
Lemma \ref{eta-0-1} gives
\begin{align*}
L_P(s,\chi_{D_1,D_2}^{(\Delta)})
&=\sum_{r=0}^{n-1}\|P\|^r x^{2r} + (\|P\|^n+\|P\|^{n-1})\sum_{r\geq n}x^{2r}\\
&=\frac{1-\|P\|^n x^{2n}}{1-\|P\| x^2} + \frac{(\|P\|^n+\|P\|^{n-1})x^{2n}}{1-x^2}\\
&=\frac{1-x^2+\|P\|^{n-1} x^{2n}(1-\|P\|^2 x^2)}
{(1-x^2)(1-\|P\| x^2)}.
\end{align*}
Next suppose that $P\mid D_1$ and $P\mid D_2$.  
Then $P\nmid E$, $n=m_P+1$, and $\alpha_P=\beta_P=0$. 
Hence Lemma \ref{eta-0-1} gives
$$L_P(s,\chi_{D_1,D_2}^{(\Delta)})=\sum_{r=0}^{n-1}\|P\|^r x^{2r}
=\frac{1-\|P\|^n x^{2n}}{1-\|P\| x^2}.$$
Finally, suppose that $P$ divides exactly one of $D_1$ and $D_2$.  
Then $P\mid E$, $n=m_P$, and exactly one of $\alpha_P$ and $\beta_P$ is zero.  
Let $\gamma_P$ be $\alpha_P$ if $P\nmid D_1$, or $\beta_P$ if $P\mid D_1$.
By Lemma \ref{eta-0-1},
\begin{align*}
L_P(s,\chi_{D_1,D_2}^{(\Delta)})
&=\sum_{r=0}^{n-1}\|P\|^r x^{2r} + \|P\|^n x^{2n}\sum_{j\geq0}(\gamma_Px)^j\\
&=\frac{1-\|P\|^n x^{2n}}{1-\|P\| x^2} + \frac{\|P\|^n x^{2n}}{1-\gamma_P x}\\
&=\frac{1-\gamma_P x-\|P\|^{n-1}x^{2n}(\|P\|x-\gamma_P)\|P\| x}
{(1-\gamma_P x)(1-\|P\| x^2)}.
\end{align*}
We note that $L_P\bigl(s,\chi_{D_1,D_2}^{(\Delta)}\bigr)=1$ if $P\mid F$ and $P\nmid F_0$; so, only the primes dividing $F_0$ contribute nontrivial $\epsilon_P$-factors. 
Combining these conductor factors \eqref{local-factor-at-conductor}
with Lemma \ref{lemma-4.1} proves \eqref{KM-thm-1-Rit-12.3.2-eq}. 

Now we show the meromorphic continuation and functional equation.
For $x=\|P\|^{-s}$, we recall that
$$\epsilon_P(s,\chi_{D_1,D_2}^{(\Delta)})
=\frac{(1-\alpha_P x)(1-\beta_P x)-\|P\|^{m_P-1}x^{2m_P}(\|P\|x-\alpha_P)(\|P\|x-\beta_P)}
{1-\|P\|x^2}.$$
The substitution $s\mapsto1-s$ sends $x$ to $(\|P\|x)^{-1}$, and a direct calculation gives
\begin{equation*}\label{local-factor-functional-equation}
\epsilon_P(1-s,\chi_{D_1,D_2}^{(\Delta)})
=\|P\|^{-m_P}x^{-2m_P}\epsilon_P(s,\chi_{D_1,D_2}^{(\Delta)})
=\|P\|^{m_P (2s-1)}\epsilon_P(s,\chi_{D_1,D_2}^{(\Delta)}).
\end{equation*}
Multiplying the above equations over $P\mid F_0$ gives
\begin{equation}\label{E-delta-functional-equation}
\|F_0\|^s \prod_{\substack{P\mid F_0\\P\mathrm{:prime}}}
\epsilon_P(s,\chi_{D_1,D_2}^{(\Delta)})
=\|F_0\|^{1-s} \prod_{\substack{P\mid F_0\\P\mathrm{:prime}}}
\epsilon_P(1-s,\chi_{D_1,D_2}^{(\Delta)}).
\end{equation}
For a square-free polynomial $D\in \cA$, put $D^\ast:=(-1)^{\deg D^{(1)}}D^{(1)}$.
By \eqref{L-JD-quad-L}, we see that
$$\Lambda_{D^\ast}(s)=G_{D^\ast}(s)L_{D^\ast}(s)=G_D(s)L(s,J_D).$$
By \eqref{quad-L-fcn-eq}, we have the functional equation
$$\Lambda_{D^\ast}(s)=\Lambda_{D^\ast}(1-s).$$
We note that the same conclusion holds also when $\deg D=0$. 
Moreover, $G_D(s)$ and $L(s,J_D)$ are periodic with period $2\tau=2\pi i/\log{q}$. Consequently, for $\kappa_{\Delta}\in\{0,1\}$,
\begin{align}\label{L-JD-kappa-fcn-eq}
G_D(s+\kappa_{\Delta} \tau)L(s+\kappa_{\Delta} \tau,J_D)
&=G_D(1-s-\kappa_{\Delta} \tau)L(1-s-\kappa_{\Delta} \tau,J_D)\nonumber\\
&=G_D(1-s+\kappa_{\Delta} \tau)L(1-s+\kappa_{\Delta} \tau,J_D).
\end{align}
Combining \eqref{E-delta-functional-equation}, \eqref{L-JD-kappa-fcn-eq} 
with \eqref{KM-thm-1-Rit-12.3.2-eq} gives the functional equation
$$\Lambda(s,\chi_{D_1,D_2}^{(\Delta)})=\Lambda(1-s,\chi_{D_1,D_2}^{(\Delta)})$$
and 
$L(s,\chi_{D_1,D_2}^{(\Delta)})$ admits meromorphic continuation to the whole complex plane.
\end{proof}

\section{Partial zeta functions and sector formulas}\label{Hayes-formula}
The purpose of this section is to adapt Hayes’ sector decomposition to the partial zeta functions appearing in the genus character $L$-function. We first express the genus character $L$-function in terms of wide or narrow partial zeta functions. We then decompose a fundamental sector into sectors determined by consecutive corners, and apply Hayes’ formulas to the sector zeta functions. 
The corner data are related to continued fractions in Subsection \ref{corner-lattice-reduced}.

Throughout this section, let $\Delta\in\cA$ be a monic nonsquare polynomial of even degree. Then $K=k(\sqrt{\Delta})$ is a real quadratic function field, and we write
$\Delta=D_KF^2$, where $D_K\in\cA^{(1)}$ is square-free and $F\in\cA^{(1)}$ is the conductor of $\cO_\Delta$. Choose $u\in\cO_\Delta$ with $\chisgn(\cN(u))=-1$ as in Proposition \ref{prop-negative-norm-element}.

\subsection{Partial zeta functions and sector zeta functions}\label{subsec-4.1} 

For a wide ideal class $\fa P(\Delta)\in C(\Delta)$ and a narrow ideal class $\fa P^+(\Delta)\in C^+(\Delta)$, we define the {\it wide and narrow partial zeta functions} 
for $\operatorname{Re}(s)>1$ by 
$$Z_0\bigl(s,\fa P(\Delta)\bigr)=\sum_{\fb \in \fa P(\Delta)}\frac{1}{\|\fN_\Delta(\fb)\|^s}, \qquad
Z_{\chisgn}\bigl(s,\fa P^+(\Delta)\bigr)=\sum_{\fb \in \fa P^+(\Delta)}\frac{1}{\|\fN_\Delta(\fb)\|^s},$$
where the sums are over $\cO_\Delta$-invertible ideals of $\cO_\Delta$. 
If $h^+(\Delta)=2h(\Delta)$, then the two narrow classes above $\fa P(\Delta)$ are
$\fa P^+(\Delta)$ and $u\fa P^+(\Delta)$, and
\begin{equation}\label{Z-chi-minus}
Z_{\chisgn}\bigl(s,\fa P^+(\Delta)\bigr)-Z_{\chisgn}\bigl(s,u\fa P^+(\Delta)\bigr)
=2Z_{\chisgn}\bigl(s,\fa P^+(\Delta)\bigr)-Z_0\bigl(s,\fa P(\Delta)\bigr).
\end{equation}

\begin{prop}\label{partial-zeta-decompositions}
Assume that $\Delta=D_1D_2F_0^2$ is as in Theorem \ref{KM-thm-3-Rit-15.1.2}.
\begin{enumerate}[(1)]
\item\label{prop-4.1.1} If $\deg D_1\equiv\deg D_2\equiv0\pmod 2$, then
$$L(s,\chi_{D_1,D_2}^{(\Delta)})=\sum_{\fa P(\Delta)\in C(\Delta)}
\chi_{D_1,D_2}^{(\Delta)}\bigl(\fa P(\Delta)\bigr)Z_0\bigl(s,\fa P(\Delta)\bigr).$$
\item\label{prop-4.1.2} If $\deg D_1\equiv\deg D_2\equiv1\pmod 2$, then
$h^+(\Delta)=2h(\Delta)$ and
$$L(s,\chi_{D_1,D_2}^{(\Delta)})=\sum_{\fa P(\Delta)\in C(\Delta)}
\chi_{D_1,D_2}^{(\Delta)}\bigl(\fa P^+(\Delta)\bigr)
\Big(Z_{\chisgn}\bigl(s,\fa P^+(\Delta)\bigr)-Z_{\chisgn}\bigl(s,u\fa P^+(\Delta)\bigr)\Big),$$
where $\fa P^+(\Delta)$ is either of the two narrow classes above $\fa P(\Delta)$.
\end{enumerate}
\end{prop}
\begin{proof}

This follows from the definition of the genus character $L$-function in \eqref{genus-char-L-def} and Proposition \ref{change-Rit-11.3.3}. Moreover, in part (2), replacing $\fa P^+(\Delta)$ by $u\fa P^+(\Delta)$ changes the signs of both the genus-character value and the parenthesized difference. Hence each summand on the right-hand side is independent of the choice of the narrow class above $\fa P(\Delta)$.
\end{proof}

For $x\in K^\times$, let 
$$\cS(x):=\frac{\bar{x}}{x}.$$ 
Let $\chi_0:k_\infty^\times \to \{1\}$ be the trivial character, 
extended to $k_\infty$ by $\chi_0(0)=0$. For $\eta\in\{\chi_0,\chisgn\}$, let
$$K_\eta:=\{x\in K^\times:\eta(\cS(x))=1\}$$
and 
$$Z_\eta(s,\fa):=
\begin{cases}
Z_0\bigl(s,\fa P(\Delta)\bigr),&\text{if }\eta=\chi_0,\\
Z_{\chisgn}\bigl(s,\fa P^+(\Delta)\bigr),&\text{if }\eta=\chisgn.
\end{cases}$$
Choose $\varepsilon_\eta\in\cO_\Delta^\times$ so that
$$K_\eta\cap\cO_\Delta^\times=\FF_q^\times\times\langle\varepsilon_\eta\rangle,\qquad
\sgn\bar{\varepsilon}_\eta=1,\qquad
\|\cS(\varepsilon_\eta)\|>1$$
(cf. \cite[(4.3)]{Hay87}).
We denote by $\varepsilon_0$ the above fundamental unit $\varepsilon_{\chi_0}$. 

We may choose a representative $\fa$ to be an $\cO_\Delta$-invertible ideal of $\cO_\Delta$.
Every ideal of $\cO_\Delta$ in 
$\fa^{-1}P(\Delta)$ or $\fa^{-1}P^+(\Delta)$ has the form $x\fa^{-1}$ with $x\in\fa \cap K_\eta$, 
and two such elements give the same ideal precisely when they differ by a unit in $\cO_\Delta^\times \cap K_\eta$. 
Since $\|\fN_\Delta(x\fa^{-1})\| = {\|\cN(x)\|}/{\|\fN_\Delta(\fa)\|}$, we have
\begin{equation}\label{Hayes-4.4-present}
Z_\eta(s,\fa^{-1})=\frac{1}{q-1}
\|\fN_\Delta(\fa)\|^s \sideset{}{'}\sum_{x\in\fa \cap K_\eta} \frac{1}{\|\cN(x)\|^s}.
\end{equation}
Here the prime means that one representative $x$ is taken from each $\langle\varepsilon_\eta\rangle$-orbit of $\fa \cap K_\eta$; the factor $q-1$ accounts for the
constant units $\FF_q^\times$.  

For $\rho,\tau\in K^\times$ satisfying $\|\cS(\rho)\|<\|\cS(\tau)\|$, 
the {\it sector} of $(\rho,\tau)$ is defined by\footnote{Under the embedding $K^\times \hookrightarrow k_\infty^\times \times k_\infty$, $x\mapsto e(x):=(x,\bar{x})$, the sector $X(\rho,\tau)$ is the pullback to $K^\times$ of Hayes' sector $X(e(\rho),e(\tau))$ (cf. \cite[(4.7)]{Hay87}).}
$$X(\rho,\tau):=\bigl\{x\in K^\times : \|\cS(\rho)\|\le \|\cS(x)\|<\|\cS(\tau)\|\bigr\}.$$
Let $\delta_\eta$ be the characteristic function of $K_\eta$. 
Since $\cS(\varepsilon_{\eta}^{n} x)=\cS(\varepsilon_\eta)^n \cS(x)$, 
for $\rho\in K^\times$, the sector
$$\mathcal{X}_\eta=X(\rho,\varepsilon_\eta\rho)$$
is a fundamental domain for the action of $\langle\varepsilon_\eta\rangle$ on $K^\times$. 
Therefore, the identity \eqref{Hayes-4.4-present} becomes
\begin{equation}\label{Hayes-4.6-present}
Z_\eta(s,\fa^{-1})
=\frac{1}{q-1}\|\fN_\Delta(\fa)\|^s \sum_{x\in \fa \cap \mathcal{X}_\eta}
\frac{\delta_\eta(x)}{\|\cN(x)\|^s}.
\end{equation}

For $\rho,\tau\in K^\times$, put $[\rho,\tau]:=\cA\rho+\cA\tau$. 
We define the {\it sector zeta function} of $X(\rho,\tau)$ by
\begin{equation}\label{sector-zeta-definition}
Z_\eta(s;\rho,\tau):=\frac{1}{q-1}
\|\bar{\rho}\tau-\rho\bar{\tau}\|^s \sum_{x\in X(\rho,\tau)\cap[\rho,\tau]} \frac{\delta_\eta(x)}{\|\cN(x)\|^s}.
\end{equation}

\subsection{Corner pairs and the sector formulas}\label{sec-corner}
In Subsection \ref{subsec-4.1}, a fundamental domain $\mathcal{X}_\eta$ for the action 
of $\langle\varepsilon_\eta\rangle$ on $K^\times$ is chosen as a sector 
$X(\rho,\varepsilon_\eta\rho)$, 
but the two boundary elements may not form a basis of an $\cO_\Delta$-invertible ideal $\fa$ of $\cO_\Delta$.
The role of a corner pair is to subdivide this sector into smaller sectors 
whose consecutive boundary elements form bases of $\fa$; so, one can represent a partial zeta function as a sum of sector zeta functions related to $\fa$.

An ordered basis $(\rho,\tau)$ of an $\cA$-module $\fa\subseteq K$ is called a {\it corner pair} if
$$\sgn\overline\rho=\sgn\overline\tau=1,\qquad
\|\rho\|>\|\tau\|,\qquad
\|\overline\rho\|<\|\overline\tau\|.$$
An element occurring in a corner pair is called a {\it corner}.

\begin{prop}[Hayes' corner decomposition]\label{Hayes-corner-decomposition}
Let $\fa$ be an $\cO_\Delta$-invertible ideal of $\cO_\Delta$.
\begin{enumerate}[(1)]
\item\label{hayes-coner-1} 
The ideal $\fa$ admits a corner pair. 
Once an initial corner pair $(\gamma_0,\gamma_1)$ is chosen, there is a unique doubly infinite sequence $(\gamma_r)_{r\in\ZZ}\subset\fa$ such that every consecutive pair $(\gamma_{r-1},\gamma_r)$ is a corner pair and so a basis of $\fa$.
\item\label{hayes-coner-2} 
Let $\varepsilon_0$ be the unit chosen in Subsection \ref{subsec-4.1}. 
There exists an integer $\partial_{\fa}>0$ such that
$$\varepsilon_0\gamma_r=\gamma_{r+\partial_{\fa}}\qquad (r\in\ZZ).$$
Choosing $\gamma_0$ as the corner on which the fundamental domain is based, 
one obtains the disjoint decomposition
\begin{equation}\label{fundamental-sector-decomposition}
\mathcal{X}_0=X(\gamma_0,\gamma_{\partial_{\fa}})
=\bigsqcup_{r=1}^{\partial_{\fa}}X(\gamma_{r-1},\gamma_r).
\end{equation}
\item\label{hayes-coner-3} If $h^+(\Delta)=2h(\Delta)$, then
$$\cO_\Delta^\times=\cO_\Delta^\times \cap K_{\chisgn}.$$
Hence one may take $\varepsilon_{\chisgn}=\varepsilon_0$ and
$\mathcal X_{\chisgn}=\mathcal X_0$; in particular, \eqref{fundamental-sector-decomposition} also gives the fundamental sector decomposition for $\eta=\chisgn$.
\end{enumerate}
\end{prop}

\begin{proof}
Via the identification $\fa\rightarrow e(\fa)$, $x\mapsto e(x):=(x,\bar x)$, the definition above is precisely Hayes' definition of a corner pair.
The existence of a corner pair is \cite[Proposition~6.9]{Hay87}. 
The properties proved in \cite[p. 214]{Hay87} give the unique doubly infinite corner sequence in \eqref{hayes-coner-1}. 
The translation of this sequence by the unit action is \cite[(6.10)]{Hay87}, 
and the sector decomposition in \cite[p. 218]{Hay87} gives \eqref{hayes-coner-2}.
Suppose that $h^+(\Delta)=2h(\Delta)$. Then every unit has positive norm
sign, which gives the equality of unit subgroups in \eqref{hayes-coner-3}.
\end{proof}

For the rest of this subsection, we assume that either $\eta=\chi_0$
or $\eta=\chisgn$ and $h^+(\Delta)=2h(\Delta)$.\footnote{If $h^+(\Delta)=h(\Delta)$, then $\varepsilon_{\chisgn}=\varepsilon_0^2$ and $\mathcal{X}_{\chisgn}=X(\gamma_0,\varepsilon_0^2 \gamma_0)=X(\gamma_0,\gamma_{2\partial_\fa})$; hence, \eqref{fundamental-sector-decomposition} does not apply to the case of $\eta=\chisgn$.} 
For every $1\le r\le\partial_{\fa}$, one has $[\gamma_{r-1},\gamma_r]=\fa$ ; so, the definition \eqref{Rit-8.5.1} of the ideal norm gives
$$\|\bar{\gamma}_{r-1}\gamma_r-\gamma_{r-1}\bar{\gamma}_r\|
=\|\sqrt\Delta\|\,\|\fN_\Delta(\fa)\|.$$
Using \eqref{Hayes-4.6-present}, \eqref{sector-zeta-definition} and the disjoint union \eqref{fundamental-sector-decomposition}, we obtain
\begin{align}\label{partial-zeta-sector-sum}
Z_\eta(s,\fa^{-1})
&=\frac{1}{q-1}\|\fN_\Delta(\fa)\|^s \sum_{r=1}^{\partial_{\fa}} \,\,
\sum_{x\in X(\gamma_{r-1},\gamma_r) \cap[\gamma_{r-1},\gamma_r]}
\frac{\delta_\eta(x)}{\|\cN(x)\|^s} \nonumber\\
&=\sum_{r=1}^{\partial_{\fa}} \frac{\|\fN_\Delta(\fa)\|^s}
{\|\bar{\gamma}_{r-1}\gamma_r-\gamma_{r-1}\bar{\gamma}_r\|^s}
Z_\eta(s;\gamma_{r-1},\gamma_r) \nonumber\\
&=\|\sqrt\Delta\|^{-s} \sum_{r=1}^{\partial_{\fa}} Z_\eta(s;\gamma_{r-1},\gamma_r).
\end{align}

It remains to evaluate each sector zeta function. 
Let $e_\eta$ be the order of $\eta$; thus $e_\eta=1$ for $\eta=\chi_0$ 
and $e_\eta=2$ for $\eta=\chisgn$. 
Since
$$\delta_\eta(x)=\frac{1}{e_\eta}\sum_{j=0}^{e_\eta-1}\eta^j(\cS(x)),$$
the identity \eqref{sector-zeta-definition} becomes
\begin{equation}\label{sector-character-average}
Z_\eta(s;\rho,\tau)=\frac{1}{q-1}\cdot\frac{1}{e_\eta}
\sum_{j=0}^{e_\eta-1}\mathcal{D}(s,\eta^j;\rho,\tau),
\end{equation}
where
\begin{equation}\label{cD-def}
\mathcal{D}(s,\eta^j;\rho,\tau) := \|\bar{\rho}\tau-\rho\bar{\tau}\|^s 
\sum_{x\in X(\rho,\tau)\cap[\rho,\tau]} \frac{\eta^j(\cS(x))}{\|\cN(x)\|^s}.
\end{equation}
For $0 \le j \le e_\eta-1$, put
$$\mathcal{E}_j(s,\eta;\fa):=\sum_{r=1}^{\partial_{\fa}} 
\mathcal{D}(s,\eta^j;\gamma_{r-1},\gamma_r).$$
Thus \eqref{partial-zeta-sector-sum} and \eqref{sector-character-average} yield
\begin{equation}\label{partial-zeta-E-sum}
Z_\eta(s,\fa^{-1})=\frac{1}{q-1}\cdot\frac{\|\sqrt{\Delta}\|^{-s}}{e_\eta}
\sum_{j=0}^{e_\eta-1}\mathcal E_j(s,\eta;\fa).
\end{equation}
Let $U=q^{-s}$ and 
\begin{equation}\label{P-function-def}
P_j(U,\eta;\fa):=\sum_{r=1}^{\partial_{\fa}}\eta^j(\gamma_r)(qU)^{l(r)},\qquad
l(r):=\deg\left(\frac{\gamma_{r-1}}{\gamma_r}\right),
\qquad j\in\{0,1\}.
\end{equation}
To compute $\mathcal{D}(s,\eta^j;\rho,\tau)$ and $\mathcal{E}_j(s,\eta;\fa)$, 
we write an element  in a sector $X(\gamma_{r-1},\gamma_{r})$ as
$$x=A\gamma_{r-1}+B\gamma_r,\qquad A,B\in\cA.$$
Hayes translates membership in the sector into four mutually exclusive conditions 
on $\deg{A}$, $\deg{B}$, $\sgn{A}$ and $\sgn{B}$ (cf. \cite[Lemma 7.5]{Hay87}). 
Moreover, Hayes' computations \cite[(7.18), (7.21)]{Hay87} give for either
$\eta=\chi_0$ or $\eta=\chisgn$ with $h^+(\Delta)=2h(\Delta)$
\begin{equation}\label{Hayes-E0-formula}
\mathcal{E}_0(s,\eta;\fa)
=\frac{(q-1)^2}{(1-qU)^2}\partial_{\fa} 
-\frac{q(q-1)(U-1)^2}{(1-qU)^2(1-qU^2)}P_0(U,\eta;\fa)
+\frac{q-1}{1-qU^2}P_0((qU)^{-1},\eta;\fa)
\end{equation}
and for $\eta=\chisgn$ with $h^+(\Delta)=2h(\Delta)$
\begin{equation}\label{Hayes-E1-formula}
\mathcal{E}_1(s,\eta;\fa)
 =-\frac{q-1}{1-qU^2}P_1(U,\eta;\fa)
 +\frac{q-1}{1-qU^2}P_1((qU)^{-1},\eta;\fa).
\end{equation}

\begin{prop}\label{Hayes-partial-zeta-formulas}
With the notation above,
\begin{align}\label{Z-0-formula}
Z_0\bigl(s,\fa^{-1}P(\Delta)\bigr)
=U^{\frac{\deg\Delta}{2}}
\bigg(&\frac{q-1}{(1-qU)^2}\partial_{\fa}
-\frac{q(U-1)^2}{(1-qU)^2(1-qU^2)} P_0(U,\chi_0;\fa)\nonumber\\
&+\frac{1}{1-qU^2}P_0((qU)^{-1},\chi_0;\fa)\bigg).
\end{align}
If $h^+(\Delta)=2h(\Delta)$, then
\begin{equation}\label{Z-chi-difference-formula}
Z_{\chisgn}\bigl(s,\fa^{-1} P^+(\Delta)\bigr)-Z_{\chisgn}\bigl(s,u\fa^{-1} P^+(\Delta)\bigr)
=\frac{U^{\frac{\deg\Delta}{2}}}{1-qU^2}
\left(-P_1(U,\chisgn;\fa)+P_1((qU)^{-1},\chisgn;\fa)\right).
\end{equation}
\end{prop}

\begin{proof}
For the trivial character $\eta=\chi_0$, \eqref{partial-zeta-E-sum} and \eqref{Hayes-E0-formula} give \eqref{Z-0-formula}.

Suppose that $h^+(\Delta)=2h(\Delta)$. 
For the narrow ideal class $\fa^{-1}P^+(\Delta)$, \eqref{partial-zeta-E-sum} gives
$$Z_{\chisgn}\bigl(s,\fa^{-1}P^+(\Delta)\bigr)
=\frac{U^{\frac{\deg\Delta}{2}}}{2(q-1)} 
\bigl(\mathcal{E}_0(s,\chisgn;\fa)+\mathcal{E}_1(s,\chisgn;\fa)\bigr).$$
We may choose $u\in\cO_\Delta$ with $\chisgn(\cN(u))=-1$ and $\sgn{\bar{u}}=1$. 
Then we have $[u\gamma_{r-1},u\gamma_r]=u\fa$ for $1\le r\le\partial_{\fa}$. 
Since $\chisgn(\cS(u))=\chisgn(\cN(u))=-1$, applying \eqref{cD-def} to $u\fa$, we get
\begin{align*}
Z_{\chisgn}\bigl(s,u\fa^{-1}P^+(\Delta)\bigr)
&=Z_{\chisgn}\bigl(s,(u\fa)^{-1}P^+(\Delta)\bigr)\\
&=\frac{U^{\frac{\deg\Delta}{2}}}{2(q-1)} 
\bigl(\mathcal{E}_0(s,\chisgn;u\fa)+\mathcal{E}_1(s,\chisgn;u\fa)\bigr)\\
&=\frac{U^{\frac{\deg\Delta}{2}}}{2(q-1)} 
\bigl(\mathcal{E}_0(s,\chisgn;\fa)-\mathcal{E}_1(s,\chisgn;\fa)\bigr).
\end{align*}
Subtracting the two identities and then applying \eqref{Hayes-E1-formula} proves
\eqref{Z-chi-difference-formula}.  
\end{proof}

\section{Continued fractions and a proof of Theorem \ref{KM-thm-3-Rit-15.1.2}}\label{sec-last-proofs}

We now identify the corner sequence of Section \ref{Hayes-formula} with the standard continued fraction of a reduced real quadratic irrational function. This allows us to evaluate the relevant partial zeta functions at $s\in\{0,\pi i/\log{q}\}$ and to prove Theorem \ref{KM-thm-3-Rit-15.1.2}. 

Throughout this section, we continue to use the assumptions and notation of Section \ref{Hayes-formula}. In particular, $\Delta$ is a monic nonsquare polynomial of even degree, and hence $\kappa_\Delta=0$.

\subsection{Continued fractions and special values of partial zeta functions}\label{corner-lattice-reduced}
Let 
$$W=\frac{B+\sqrt{\Delta}}{2A}\in\XX_\Delta^0$$ 
be a reduced real quadratic irrational function of discriminant $\Delta$ with type $(A,B,C)$ satisfying $\sgn\overline{W}=1$. Put
$$\sigma_W:=\sgn(B-\sqrt\Delta)=\sgn(2A).$$
Define
$$\gamma_0=\frac{B+\sqrt\Delta}{\sigma_W},\qquad
\gamma_1=\frac{2A}{\sigma_W},\qquad
\fa=[\gamma_0,\gamma_1]=[2A,B+\sqrt\Delta],\qquad
\nu_A=\frac{1-\chisgn(A)}{2}.$$
Then $\sgn(\bar{\gamma}_0)=\sgn(\bar{\gamma}_1)=1$. 
By Proposition \ref{app-real-correspondence}, we see 
$\iota_\Delta([W]_{\sim_+})=u^{\nu_A} \fa P^+(\Delta)$. 
Since
$$\left\|\frac{\gamma_0}{\gamma_1}\right\|=\|W\|>1,\qquad
\left\|\frac{\bar{\gamma}_0}{\bar{\gamma}_1}\right\|
 =\|\overline W\|<1,$$
$(\gamma_0,\gamma_1)$ is a corner pair of $\fa$. 
We note that 
$$\chisgn\left(\frac{\bar{\gamma}_0\gamma_1
-\gamma_0\bar{\gamma}_1}{\sqrt\Delta}\right)
=\chisgn\left(-\frac{4A}{\sigma_W^2}\right)=\chisgn(-A)$$
and the ordered basis $(\gamma_0,\gamma_1)$ of $\fa$ is not required to be oriented.

Following Hayes' continued fraction algorithm \cite[(5.4) and (6.7)]{Hay87}, 
put $W_1^H=W$ and define recursively
\begin{equation}\label{Hayes-recursion}
W_r^H=M_r^H+\frac{\zeta_r}{W_{r+1}^H},\qquad
M_r^H=\lfloor W_r^H\rfloor,\qquad
\sgn{\overline{W_{r+1}^H}}=1,\end{equation}
where $\zeta_r\in\FF_q^\times$ is uniquely determined by the last condition. 
We call $[M_1^H,M_2^H,M_3^H,\ldots]$ {\it Hayes' continued fraction} of $W$.

Let $W_r$ and $M_r=\lfloor W_r\rfloor$ be the $r$-th standard complete and partial quotients of $W$, respectively (cf. Appendix \ref{app:continued-fractions}). 
By induction, $W_r/W_r^H\in\FF_q^\times$ for every $r\ge1$. 
Define $\mu_r\in\FF_q^\times$ by
\begin{equation}\label{def-mu-r}
W_r=\mu_rW_r^H,
\qquad \mu_1=1.\end{equation}
Then \eqref{Hayes-recursion} and the standard recursion \eqref{l-l-comparision} give
$W_{r+1}=(\mu_r\zeta_r)^{-1}W_{r+1}^H$; so,
\begin{equation}\label{zeta-mu}
\zeta_r=\frac{1}{\mu_r\mu_{r+1}}.
\end{equation}
Taking conjugates in \eqref{Hayes-recursion} and using the reducedness of $W_r^H$ by Proposition \ref{reduced-cf-periodicity} gives
\begin{equation*}
M_r^H=-\zeta_r
 \left\lfloor\frac{1}{\overline{W_{r+1}^H}}\right\rfloor.
\end{equation*}
Consequently, the last condition in \eqref{Hayes-recursion} gives
\begin{equation}\label{sgn-W-minus-zeta}
\sgn W_r^H=\sgn M_r^H=-\zeta_r.
\end{equation}

We define a sequence $\{\gamma_r\}_{r\ge0}$ in $\cO_\Delta$ recursively by 
$$W_r^H=\frac{\gamma_{r-1}}{\gamma_r}\qquad (r\ge1).$$
Then $(\gamma_{r-1},\gamma_r)$ is a corner pair of $\fa$, and the positive integer
$\partial_\fa$ in Proposition \ref{Hayes-corner-decomposition} is the minimal period of 
$(W_r^H)_{r\in\NN}$ (cf. \cite[Proposition 6.13]{Hay87}). 
It follows that
$$\prod_{r=1}^{\partial_{\fa}}W_r^H=\frac{\gamma_0}{\gamma_{\partial_\fa}}=\frac{1}{\varepsilon_0}.$$

\begin{lem}\label{chi-gamma-formula}
Let $\partial_W$ be the quasi-period of $W$ (cf. Appendix \ref{app:continued-fractions}). 
Then
$$\partial_{\fa}=\partial_W,\qquad l(r)=\deg M_r\qquad (1\le r\le\partial_W),$$
where $l(r)$ is defined in \eqref{P-function-def}.
Moreover, for every $r\ge1$,
\begin{equation}\label{chi-gamma-standard}
\chisgn(\gamma_r)=\chisgn(W)\chisgn(-1)^{r-1}\chisgn(M_r).
\end{equation}
\end{lem}
\begin{proof}
Since $W_r/W^H_r\in\FF_q^\times$ for all $r\ge1$, the minimal period $\partial_{\fa}$  of 
$(W_r^H)_{r\in\NN}$ is greater than or equal to the quasi-period $\partial_W$. 
Conversely, write $W_1=\alpha W_{1+\partial_W}$ with
$\alpha\in\FF_q^\times$.  The standard recursion \eqref{l-l-comparision} gives
$W_r=c_rW_{r+\partial_W}$, where $c_r=\alpha$ for odd $r$ and $c_r=\alpha^{-1}$ for even $r$. 
Since $\sgn{\overline{W_r^H}}=\sgn{\overline{W_{r+\partial_W}^H}}=1$, 
taking conjugates in \eqref{def-mu-r} gives
$\mu_r=c_r\mu_{r+\partial_W}$, and hence
$$W_r^H=\mu_r^{-1}W_r=c_r^{-1}\mu_{r+\partial_W}^{-1}W_{r}
=\mu_{r+\partial_W}^{-1}W_{r+\partial_W}=W_{r+\partial_W}^H.$$  
Thus $\partial_{\fa}\le\partial_W$ and so $\partial_{\fa}=\partial_W$ .  
By \eqref{l-l-comparision} and the reducedness of $W_{r+1}$, we have 
$$l(r)=\deg{W_r^H}=\deg{W_r}=\deg{M_r}.$$
By \eqref{zeta-mu} and $\mu_1=1$, we get 
$$\prod_{n=1}^{r-1}\chisgn(\zeta_n)=\chisgn(\mu_{r}).$$
We see $\chisgn(W)=\chisgn(\gamma_0)$. 
Then telescoping the complete quotients $W_n^H=\gamma_{n-1}/\gamma_n$ 
together with \eqref{sgn-W-minus-zeta} gives
\begin{equation*}
\chisgn(W)\chisgn(\gamma_r)
=\prod_{n=1}^{r}\chisgn\left(\frac{\gamma_{n-1}}{\gamma_n}\right)
=\chisgn(-1)^{r-1}\chisgn(\mu_r)\chisgn\left(W_r^H\right)
=\chisgn(-1)^{r-1}\chisgn(M_r).\qedhere
\end{equation*}
\end{proof}

Applying Lemma \ref{chi-gamma-formula} to \eqref{P-function-def} gives
\begin{align}
\label{P-function-standard-0}
P_0(U,\chi_0;\fa) &= \sum_{r=1}^{\partial_W}(qU)^{\deg M_r},\\
\label{P-function-standard-1}
P_1(U,\chisgn;\fa) &= \chisgn(W)\sum_{r=1}^{\partial_W}
\chisgn\left((-1)^{r-1}M_r\right)(qU)^{\deg M_r}.
\end{align}
Let 
$$\tau:=\frac{\pi i}{\log q},\qquad t:=\#\{r:1\le r\le\partial_W,\, \deg{M_r}\text{ is odd}\}.$$

\begin{prop}\label{Hayes-special-values}
With the notation above,
\begin{equation}\label{even-Hayes}
Z_0\bigl(\tau,\fa^{-1}P(\Delta)\bigr)
=\frac{(-1)^{\frac{\deg\Delta}{2}}}{(q+1)^2} \left(4q \sum_{r=1}^{\partial_W} 
(-1)^{\deg M_r} \frac{\|M_r\|-1}{q-1} + 2t(q-1)\right).
\end{equation}
If $h^+(\Delta)=2h(\Delta)$, then\footnote{The right-hand side of \eqref{odd-zero-Hayes} is twice Hayes' formula \cite[(8.3)]{Hay87}, because of \eqref{Z-chi-minus} and $Z_0\bigl(0,\fa^{-1}P(\Delta)\bigr)=0$ by \eqref{Z-0-formula} at $U=1$.}
\begin{equation}\label{odd-zero-Hayes}
Z_{\chisgn}\bigl(0,\fa^{-1} P^+(\Delta)\bigr)-Z_{\chisgn}\bigl(0,u\fa^{-1} P^+(\Delta)\bigr)
=\sum_{r=1}^{\partial_W}\chisgn(\gamma_r) \frac{\|M_r\|-1}{q-1}
\end{equation}
and
\begin{equation}\label{odd-pi-Hayes}
Z_{\chisgn}\bigl(\tau,\fa^{-1} P^+(\Delta)\bigr)
-Z_{\chisgn}\bigl(\tau,u\fa^{-1} P^+(\Delta)\bigr)
=(-1)^{\frac{\deg\Delta}{2}} \sum_{r=1}^{\partial_W}
\chisgn(\gamma_r)(-1)^{\deg M_r} \frac{\|M_r\|-1}{q-1}.
\end{equation}
\end{prop}
\begin{proof}
At $s=\tau$, one has $U=-1$. 
Substituting this value into \eqref{Z-0-formula} and using \eqref{P-function-standard-0} gives
\begin{align*}
Z_0\bigl(\tau,\fa^{-1}P(\Delta)\bigr)
=\frac{(-1)^{\frac{\deg\Delta}{2}}}{q-1} 
\left(\frac{(q-1)^2}{(q+1)^2}\partial_W 
+\frac{4q}{(q+1)^2}\sum_{r=1}^{\partial_W}(-q)^{\deg M_r}
-\sum_{r=1}^{\partial_W}(-1)^{\deg M_r}\right).
\end{align*}
Since $\partial_W-\sum_{r=1}^{\partial_W}(-1)^{\deg M_r}=2t$,
the above identity gives \eqref{even-Hayes}. 

At $s=0$ and $s=\tau$, one has $U=1$ and $U=-1$, respectively.
Substituting these values into \eqref{Z-chi-difference-formula} and using \eqref{P-function-standard-1} yields \eqref{odd-zero-Hayes} and \eqref{odd-pi-Hayes}.  
\end{proof}

\subsection{Proof of Theorem \ref{KM-thm-3-Rit-15.1.2}}\label{last-proofs}
\begin{lem}\label{class-number-special-values}
Let $D_1,D_2\in\cA$ be monic square-free polynomials, and put $\tau=\pi i/\log{q}$.
\begin{enumerate}[(1)]
\item\label{lem-5.3-1} If $\deg D_1\equiv\deg D_2\equiv0\pmod2$, then
\begin{equation*}
\left\{\begin{array}{rcll}
h(gD_1)h(gD_2)
 &=&L(\tau,J_{D_1})L(\tau,J_{D_2}) &\text{for }\deg{D_1},\,\deg{D_2}>0,\\[2mm]
h(gD)\dfrac{1}{q+1} 
&=&L(\tau,J_D)\zeta_\cA(\tau) &\text{for }\{D_1,D_2\}=\{D,1\} \text{ and }\deg{D}>0.
\end{array}\right.
\end{equation*}
\item\label{lem-5.3-2} If $\deg D_1\equiv\deg D_2\equiv1\pmod2$, then
\begin{equation*}
\left\{\begin{array}{rcl}
h(D_1)h(D_2)
&=&L\big(\tfrac{1-\chisgn(-1)}{2}\tau,J_{D_1}\big)L\big(\tfrac{1-\chisgn(-1)}{2}\tau,J_{D_2}\big),\\[2mm]
h(gD_1)h(gD_2)
&=&L\big(\tfrac{1+\chisgn(-1)}{2}\tau,J_{D_1}\big)L\big(\tfrac{1+\chisgn(-1)}{2}\tau,J_{D_2}\big).
\end{array}\right.
\end{equation*}
\end{enumerate}
\end{lem}

\begin{proof}
Let $D\in\cA^{(1)}$ be square-free with $\deg D>0$.  If $\deg D$ is
even, then $L(s,J_D)=L_D(s)$ by \eqref{L-JD-quad-L}, while \eqref{D-gD-input} gives
$L_{gD}(s)=L_D(s+\tau)$.  Since $gD$ is imaginary, the class number formula \eqref{Rit-9.3.1-new} gives $h(gD)=L_{gD}(0)=L_D(\tau)=L(\tau,J_D)$.
Moreover, $L(s,J_1)=\zeta_\cA(s)$ and
$\zeta_\cA(\tau)=1/(q+1)$ give the second identity in the case \eqref{lem-5.3-1}. 

Suppose that $\deg D$ is odd.  By \eqref{L-JD-quad-L} and \eqref{D-gD-input},
$L(s,J_D)=L_D(s+\tfrac{1-\chisgn(-1)}{2}\tau)$.
Since both $D$ and $gD$ are imaginary, the class number formula \eqref{Rit-9.3.1-new} and the $2\tau$-periodicity of $L(s,J_D)$ give
$h(D)=L_D(0)=L(-\tfrac{1-\chisgn(-1)}{2}\tau,J_D)=L(\tfrac{1-\chisgn(-1)}{2}\tau,J_D)$ and 
$h(gD)=L_{gD}(0)=L_D(\tau)=L(\frac{1+\chisgn(-1)}{2}\tau,J_D)$. 
\end{proof}

\begin{proof}[Proof of Theorem \ref{KM-thm-3-Rit-15.1.2}]
We first show that the expressions on the right-hand side of 
Theorem \ref{KM-thm-3-Rit-15.1.2} are unchanged when a reduced representative is multiplied by a nonzero constant.
Let $W\in\XX_\Delta^0$, $\lambda\in\FF_q^\times$ and $\widetilde W=\lambda W$.
If $\widetilde W_r$ and $\widetilde M_r$ denote the standard complete and partial quotients of $\widetilde W$, then a direct induction gives
$\widetilde W_r=\lambda^{(-1)^{r-1}}W_r$ and $\widetilde M_r=\lambda^{(-1)^{r-1}}M_r$; so, $\partial_{\widetilde W}=\partial_W$ and $\deg\widetilde M_r=\deg M_r$. 
Moreover,
$$\chisgn\bigl((-1)^{r-1}\widetilde M_r\bigr)
=\chisgn(\lambda)\chisgn\bigl((-1)^{r-1}M_r\bigr),\qquad
\chi_{D_1,D_2}^{(\Delta)}(\lambda W)
=\chisgn(\lambda)^{\deg D_1}\chi_{D_1,D_2}^{(\Delta)}(W).$$
Thus the expressions on the right-hand sides of Theorem~\ref{KM-thm-3-Rit-15.1.2} are unchanged under $W\mapsto\lambda W$. 

By Proposition \ref{last-remark-thm-1.4}, we may choose the
representatives $W^{(j)}$ of $\XX_\Delta^0/\sim$ with $\sgn{\overline{W^{(j)}}}=1$ in order to apply results in Subsection \ref{corner-lattice-reduced}. Write
$$W^{(j)}=\frac{B^{(j)}+\sqrt{\Delta}}{2A^{(j)}},\qquad \fa_j=[2A^{(j)},B^{(j)}+\sqrt\Delta].$$

Suppose first that $\deg D_1\equiv\deg D_2\equiv0\pmod2$.
As $j$ ranges from $1$ to $h(\Delta)$, the wide classes $\fa_j^{-1}P(\Delta)$ range over $C(\Delta)$. 
Since $\chi_{D_1,D_2}^{(\Delta)}$ is quadratic and descends to $C(\Delta)$ 
by Proposition \ref{change-Rit-11.3.3}, 
$$\chi_{D_1,D_2}^{(\Delta)}\bigl(\fa_j^{-1}P(\Delta)\bigr)
=\chi_{D_1,D_2}^{(\Delta)}\bigl(\fa_jP(\Delta)\bigr)
=\chi_{D_1,D_2}^{(\Delta)}(W^{(j)}).$$
Hence Propositions \ref{partial-zeta-decompositions} and \ref{Hayes-special-values} give
\begin{align*}
L\bigl(\tau,\chi_{D_1,D_2}^{(\Delta)}\bigr)
&=\sum_{j=1}^{h(\Delta)}\chi_{D_1,D_2}^{(\Delta)}\bigl(\fa_j^{-1}P(\Delta)\bigr)
Z_0\bigl(\tau,\fa_j^{-1}P(\Delta)\bigr)\\
&=\frac{(-1)^{\frac{\deg\Delta}{2}}}{(q+1)^2}\sum_{j=1}^{h(\Delta)}
\chi_{D_1,D_2}^{(\Delta)}(W^{(j)})\left(
4q\sum_{r=1}^{\partial^{(j)}} (-1)^{\deg M_r^{(j)}} \frac{\|M_r^{(j)}\|-1}{q-1}
+2t^{(j)}(q-1)\right).
\end{align*}
On the other hand, \eqref{KM-thm-1-Rit-12.3.2-eq} and Lemma~\ref{class-number-special-values} \eqref{lem-5.3-1} give
$$L\bigl(\tau,\chi_{D_1,D_2}^{(\Delta)}\bigr)
=\begin{cases}
\displaystyle h(gD_1)h(gD_2)\prod_{\substack{P\mid F_0\\P\mathrm{:prime}}}
\epsilon_P \left(\tau,\chi_{D_1,D_2}^{(\Delta)}\right) 
&\text{ if }D_1\neq1\text{ and }D_2\neq1,\\
\displaystyle \frac{1}{q+1}h(gD)\prod_{\substack{P\mid F_0\\P\mathrm{:prime}}}
\epsilon_P \left(\tau,\chi_{D,1}^{(\Delta)}\right) &\text{ if }\{D_1,D_2\}=\{D,1\}.
\end{cases}$$
Comparing the two expressions proves Theorem \ref{KM-thm-3-Rit-15.1.2} \eqref{thm1.4-1}.

Suppose now that $\deg D_1\equiv\deg D_2\equiv1\pmod2$.
By Proposition~\ref{change-Rit-11.3.3}, $h^+(\Delta)=2h(\Delta)$ and
$$C^+(\Delta)=\left\{\fa_jP^+(\Delta),\, u\fa_jP^+(\Delta):1\leq j\leq h(\Delta)\right\},$$ 
where $u\in\cO_\Delta$ is chosen as in Proposition \ref{prop-negative-norm-element}.
Proposition~\ref{app-real-correspondence} gives
$$\iota_\Delta([W^{(j)}]_{\sim_+})=u^{\nu_{A^{(j)}}} \fa_j P^+(\Delta),\qquad
\nu_{A^{(j)}}=\frac{1-\chisgn(A^{(j)})}{2}.$$
It follows from Proposition \ref{change-Rit-11.3.3} that
$$\chi_{D_1,D_2}^{(\Delta)}\bigl(\fa_j^{-1} P^+(\Delta)\bigr)
=\chi_{D_1,D_2}^{(\Delta)}\bigl(\fa_j P^+(\Delta)\bigr)
=(-1)^{\nu_{A^{(j)}}} \chi_{D_1,D_2}^{(\Delta)}(W^{(j)})
=\chisgn(A^{(j)}) \chi_{D_1,D_2}^{(\Delta)}(W^{(j)}).$$
Combining this identity with \eqref{chi-gamma-standard} and Proposition~\ref{last-remark-thm-1.4}, we obtain
\begin{align*}
\chi_{D_1,D_2}^{(\Delta)}\bigl(\fa_j^{-1}P^+(\Delta)\bigr)\chisgn(\gamma_r^{(j)})
&=\chi_{D_1,D_2}^{(\Delta)}(W^{(j)})\chisgn(A^{(j)}) 
\chisgn(W^{(j)})\chisgn(-1)^{r-1}\chisgn(M_r^{(j)})\\
&=\chi_{D_1,D_2}^{(\Delta)}(W^{(j)})\chisgn(-1)^{r-1}\chisgn(M_r^{(j)}).
\end{align*}
Propositions \ref{partial-zeta-decompositions} and
\ref{Hayes-special-values} therefore give
\begin{align*}
L(0,\chi_{D_1,D_2}^{(\Delta)})
&=\sum_{j=1}^{h(\Delta)} \chi_{D_1,D_2}^{(\Delta)}\bigl(\fa_j^{-1} P^+(\Delta)\bigr)
\Big(Z_{\chisgn}\bigl(0,\fa_j^{-1} P^+(\Delta)\bigr)
-Z_{\chisgn}\bigl(0,u\fa_j^{-1} P^+(\Delta)\bigr)\Big)\\
&=\sum_{j=1}^{h(\Delta)} \chi_{D_1,D_2}^{(\Delta)}\bigl(W^{(j)}\bigr) 
\sum_{r=1}^{\partial^{(j)}}\chisgn(-1)^{r-1}\chisgn(M_r^{(j)}) 
\frac{\|M_r^{(j)}\|-1}{q-1}
\end{align*}
and
\begin{align*}
L(\tau,\chi_{D_1,D_2}^{(\Delta)})
&=\sum_{j=1}^{h(\Delta)} \chi_{D_1,D_2}^{(\Delta)}\bigl(\fa_j^{-1} P^+(\Delta)\bigr)
\Big(Z_{\chisgn}\bigl(\tau,\fa_j^{-1} P^+(\Delta)\bigr)
-Z_{\chisgn}\bigl(\tau,u\fa_j^{-1} P^+(\Delta)\bigr)\Big)\\
&=(-1)^{\frac{\deg\Delta}{2}}
\sum_{j=1}^{h(\Delta)} \chi_{D_1,D_2}^{(\Delta)}\bigl(W^{(j)}\bigr) 
\sum_{r=1}^{\partial^{(j)}}\chisgn(-1)^{r-1}\chisgn(M_r^{(j)}) 
(-1)^{\deg M_r^{(j)}} \frac{\|M_r^{(j)}\|-1}{q-1}.
\end{align*}
On the other hand, \eqref{KM-thm-1-Rit-12.3.2-eq} and
Lemma \ref{class-number-special-values} \eqref{lem-5.3-2} give
$$L\left(\frac{1-\chisgn(-1)}{2}\tau, \chi_{D_1,D_2}^{(\Delta)}\right)
=h(D_1)h(D_2)\prod_{\substack{P\mid F_0\\P\mathrm{:prime}}}
\epsilon_P\left(\frac{1-\chisgn(-1)}{2}\tau,\chi_{D_1,D_2}^{(\Delta)}\right)$$
and
$$L\left(\frac{1+\chisgn(-1)}{2}\tau,\chi_{D_1,D_2}^{(\Delta)}\right)
=h(gD_1)h(gD_2)\prod_{\substack{P\mid F_0\\P\mathrm{:prime}}}
\epsilon_P\left(\frac{1+\chisgn(-1)}{2}\tau,\chi_{D_1,D_2}^{(\Delta)}\right).$$
Comparing these identities with the preceding evaluations of
$L(0,\chi_{D_1,D_2}^{(\Delta)})$ and $L(\tau,\chi_{D_1,D_2}^{(\Delta)})$ 
proves Theorem \ref{KM-thm-3-Rit-15.1.2} \eqref{thm1.4-2}.
\end{proof}

\smallskip

\appendix

\section{Primitive quadratic forms and real quadratic irrational functions}\label{app:quadratic-background}

In this appendix, we collect the form–ideal correspondence and the continued-fraction facts used in the construction of genus characters and 
in the proof of Theorem \ref{KM-thm-3-Rit-15.1.2}. 
We continue to use the notation of Sections \ref{intro-sec} and \ref{preliminary-sec}. 
Put
$$\mathrm{GL}_2(\cA)=\{M\in\mathrm{Mat}_2(\cA):\det M\in\FF_q^\times\}, \qquad
\mathrm{SL}_2(\cA)=\{M\in\mathrm{GL}_2(\cA):\chisgn(\det M)=1\}.$$

\subsection{Primitive quadratic forms and ideal classes}\label{app:form-ideal}
Let $\Delta\in\cA\setminus\cA^2$.  A primitive binary quadratic form of
 discriminant $\Delta$ is
$$f(X,Y)=[A,B,C]=AX^2+BXY+CY^2,\qquad (A,B,C)=1,\qquad B^2-4AC=\Delta,$$
where $A,B,C\in\cA$. 
For $\fM=\begin{pmatrix}a&b\\c&d\end{pmatrix}\in\mathrm{GL}_2(\cA)$, set
\begin{equation}\label{proper-form-action}
\fM\cdot f(X,Y):=(\det\fM)^{-1}f(dX+bY,cX+aY).
\end{equation}
This is again a primitive form of discriminant $\Delta$. 
Let $\mathcal{F}_\Delta$ be the set of all $\mathrm{GL}_2(\cA)$-classes of primitive forms of discriminant $\Delta$.

\begin{prop}\label{prop-negative-norm-element}
There exists $u\in\cO_\Delta$ such that $\chisgn(\cN(u))=-1$ if and only if either
\begin{enumerate}[(1)]
\item $\deg\Delta$ is even, or
\item $\deg\Delta$ is odd and $\chisgn(-\Delta)=-1$.
\end{enumerate}
\end{prop}

\begin{proof}
Suppose first that $\deg\Delta=2d$ is even. 
There exist $c,t\in\FF_q$ such that $c^2-\sgn(\Delta)t^2$ is a nonsquare in $\FF_q^\times$:
if $\sgn(\Delta)$ is a square, the form
$X^2-\sgn(\Delta)Y^2$ splits over $\FF_q$ and represents every element of $\FF_q$; 
if $\sgn(\Delta)$ is a nonsquare, it is the norm map of $\FF_{q^2}/\FF_q$, which is surjective. Let $u=cT^d+t\sqrt\Delta\in\cO_\Delta$.
The leading coefficient of $\cN(u)=c^2T^{2d}-t^2\Delta$ is $c^2-\sgn(\Delta)t^2$. 
Hence $\chisgn(\cN(u))=-1$.

Now suppose that $\deg\Delta$ is odd. 
For $z=x+y\sqrt\Delta\in \cO_\Delta$, one has $\cN(z)=x^2-y^2\Delta$. 
Since the degrees of $x^2$ and $y^2\Delta$ have opposite parity, we have $\chisgn(\cN(z))=1$ or $\chisgn(-\Delta)$ for every nonzero $z\in\cO_\Delta$.  
If $\chisgn(-\Delta)=-1$, then $u=\sqrt{\Delta}$ satisfies $\chisgn(\cN(u))=-1$.
\end{proof}

Proposition \ref{prop-negative-norm-element} remains valid with $u\in\cO_\Delta$ replaced by $u\in k(\sqrt{\Delta})$, since any $z\in k(\sqrt{\Delta})$ can be multiplied by a nonzero element of $\cA$ to obtain an element of $\cO_\Delta$, and multiplication by such an element does not change the norm sign.

If no element of negative norm sign exists, then $\chisgn(\cN(z))=1$ for every nonzero $z\in \cO_\Delta$. 
Then $4Af(x,y)=\cN\bigl(2Ax+(B+\sqrt\Delta)y\bigr)$ with $x,y\in\cA$
shows that for $(x,y)\neq(0,0)$
$$\chisgn(f(x,y))=\chisgn(A).$$ 
If $\fM=\begin{pmatrix}a&b\\c&d\end{pmatrix}\in\mathrm{SL}_2(\cA)$, the coefficient of $X^2$ in $\fM\cdot f$ is $(\det\fM)^{-1}f(d,c)$, 
and this coefficient has sign character $\chisgn(A)$.

Now we define $\mathcal{F}_\Delta^+$ by the following two cases.
\begin{enumerate}[(1)]
\item If there exists $u\in \cO_\Delta$ with
$\chisgn(\cN(u))=-1$, then $\mathcal F_\Delta^+$ is the set of all $\mathrm{SL}_2(\cA)$-classes of primitive forms of discriminant $\Delta$.
\item If no such $u$ exists, then $\mathcal F_\Delta^+$ is the set of $\mathrm{SL}_2(\cA)$-classes of primitive forms $[A,B,C]$ of discriminant $\Delta$ satisfying
$\chisgn(A)=1$.
\end{enumerate}

For a basis $w=(w_1,w_2)$ of an $\cO_\Delta$-invertible fractional 
$\cO_\Delta$-ideal $\fa$, put
$$\Omega(w):=\frac{\overline w_1w_2-w_1\overline w_2}{\sqrt\Delta},\qquad f_w(X,Y):=\frac{\cN(Xw_1+Yw_2)}{\Omega(w)}.$$
We call $f_w$ the quadratic form associated with the basis $w$.  
We write $[f_w]^+$ for its $\mathrm{SL}_2(\cA)$-class, when $w$ is oriented, 
that is, $\chisgn(\Omega(w))=1$.

\begin{prop}\label{form-ideal-correspondence-prop}
The assignment
$$\Theta^+:C^+(\Delta)\xrightarrow{\sim}\mathcal F_\Delta^+,\qquad 
\fa P^+(\Delta) \mapsto [f_w]^+,$$
where $w$ is an oriented basis of $\fa$, is a bijection.  
Without an orientation condition, the same construction gives a bijection
$$\Theta:C(\Delta)\xrightarrow{\sim}\mathcal F_\Delta.$$
\end{prop}
\begin{proof}
Every $\cO_\Delta$-invertible fractional $\cO_\Delta$-ideal has an oriented basis: if a basis is not oriented, multiplying one basis vector by $g$ reverses its orientation.  
For $\fM=\begin{pmatrix}a&b\\c&d\end{pmatrix}$, put
$\fM^\vee=\begin{pmatrix}d&b\\c&a\end{pmatrix}$ 
and define
$$\fM\cdot w:=w\fM^\vee=(dw_1+cw_2,bw_1+aw_2).$$
A direct calculation gives
\begin{equation}\label{form-action-properties}
\Omega(\fM\cdot w)=\det(\fM)\Omega(w),\qquad 
f_{\fM\cdot w}=\fM\cdot f_w,\qquad
\Omega(Zw)=\cN(Z)\Omega(w),\qquad 
f_{Zw}=f_w,
\end{equation}
where $Z\in k(\sqrt\Delta)^\times$. 
Thus an oriented change of basis corresponds exactly to an \(\mathrm{SL}_2(\cA)\)-action. Since $f_{Zw}=f_w$, 
we may multiply an $\cO_\Delta$-invertible fractional $\cO_\Delta$-ideal $\mathfrak a$ 
by some $Z\in\mathcal O_\Delta\setminus\{0\}$ and assume that $\fa\subset\cO_\Delta$. 
By a $\mathrm{GL}_2(\cA)$-action and \eqref{basis-of-the-form}, it admits a basis of the form 
$$w=(2SA,S(B+\sqrt{\Delta})),\qquad B^2-4AC=\Delta,\qquad (A,B,C)=1,\qquad
A,B,C,S\in \cA.$$
Hence
\begin{equation}\label{Qw-def}
\Omega(w)=4S^2A,\qquad f_w=[A,B,C],
\end{equation}
and so $f_w$ is a primitive form of discriminant $\Delta$.
Moreover, if $\chisgn(\cN(Z))=1$ and $w$ is oriented, then $Zw$ is oriented and 
$[f_{Zw}]^+=[f_w]^+$.
Hence the resulting form class depends only on the narrow ideal class.
If every element $u\in k(\sqrt{\Delta})^\times$ satisfies $\chisgn(\cN(u))=1$ 
and $w$ is oriented, then the coefficient of $X^2$ in $f_w$ is ${\cN(w_1)}/{\Omega(w)}$
and its sign character is $1$.  
Thus $\Theta^+:C^+(\Delta)\to\mathcal F_\Delta^+$ is well-defined.

Suppose that $f_{w'}=\fM\cdot f_w$ with $\fM\in\mathrm{SL}_2(\cA)$. 
Let $v=\fM\cdot w$ and write $f_{w'}=f_{v}=[A,B,C]$. Then
$$A=\frac{\cN(v_1)}{\Omega(v)},\qquad
B=\frac{\overline v_1v_2+v_1\overline v_2}{\Omega(v)},\qquad
\sqrt{\Delta}=\frac{\overline v_1v_2-v_1\overline v_2}{\Omega(v)}.$$
Adding the last two identities and dividing by the first identity gives
$v_2/v_1=(B+\sqrt{\Delta})/(2A)$.
The same calculation for $w'$ yields
$w'_2/w'_1=v_2/v_1$.
Thus $w'=Zv$ for some $Z\in k(\sqrt\Delta)^\times$.  Since both bases
are oriented, we have $\chisgn(\cN(Z))=1$; so, the two ideals are narrowly
equivalent.  This proves injectivity.

Conversely, let $f=[A,B,C]$ be a primitive form whose class belongs to $\mathcal F_\Delta^+$, and let
$$\fa_f=[2A,B+\sqrt\Delta].$$
Since $(A,B,C)=1$, we have
$\frac{1}{2A}\fa_f\overline{\fa}_f=[2A,B+\sqrt\Delta,B-\sqrt\Delta,2C]
=\cO_\Delta$.
Hence $\fa_f$ is an $\cO_\Delta$-invertible ideal of $\cO_\Delta$. 
For the basis $w=(2A,B+\sqrt\Delta)$, we have $f_w=f$ and $\Omega(w)=4A$. 
If there exists $u\in \cO_\Delta$ with $\chisgn(\cN(u))=-1$, then $w$ is already oriented when $\chisgn(A)=1$, while $uw$ is oriented when $\chisgn(A)=-1$; in both cases the associated quadratic form is $f$. If no such $u$ exists, the definition of $\mathcal F_\Delta^+$ gives that $\chisgn(A)=1$ and $w$ is already oriented. 
Hence $\Theta^+$ is surjective.

Without orientations, the preceding argument works with $\mathrm{GL}_2(\cA)$.
\end{proof}

We show a compatibility of the group structures under the form-ideal correspondence. 
\begin{lem}\label{qf-coefficient-prime}
Let $H\in\cA$ be a nonzero polynomial. 
For every primitive quadratic form $f$ of discriminant $\Delta$, there exists $x\in\cA$ such that $(f(x,1),H)=1$. 
Furthermore, the $\mathrm{SL}_2(\cA)$-class of $f$ contains a primitive form $[A,B,C]$ satisfying $(A,H)=1$.
\end{lem}
\begin{proof}
Let $f$ be a primitive form, and let $P$ run over the prime divisors of $H$.  
Since $f$ is primitive, the polynomial $f(X,1)$ is nonzero modulo $P$. 
Then one can choose $x_P\in\cA$ with $f(x_P,1)\not\equiv0\pmod P$ 
because its degree is at most $2$ and $|\cA/P|=\|P\|\geq3$. 
By the Chinese remainder theorem, we can choose $x\in\cA$ with $x\equiv x_P\pmod P$ for all $P\mid H$. Then $(f(x,1),H)=1$. 
Consequently, the coefficient of $X^2$ in
$\begin{pmatrix}0&-1\\1&x\end{pmatrix}\cdot f=f(xX-Y,X)$
is $f(x,1)$.
\end{proof}

\begin{prop}\label{prime-to-delta-representative}
Let $\fa_1 P^+(\Delta)$, $\fa_2 P^+(\Delta) \in C^+(\Delta)$.  One can choose representatives $\fa_i$, their oriented bases $w_i$, 
and a common middle coefficient $B\in\cA$ such that 
$$f_{w_i}=f_i=[A_i,B,C_i],\qquad (A_1A_2,\Delta)=1,\qquad (A_1,A_2)=1.$$
Furthermore, let
\begin{equation*}
f_3=[A_1A_2,B,C_3],\qquad C_3=\frac{B^2-\Delta}{4A_1A_2}.
\end{equation*}
Then $f_3$ is primitive, and the product ideal $\fa_1\fa_2$ has an oriented basis $w_3$ satisfying 
$$f_{w_3}=f_3.$$
\end{prop}
\begin{proof}
Applying Lemma \ref{qf-coefficient-prime} first to $\Theta^+(\fa_1 P^+(\Delta))$ with $H=\Delta$, and then to $\Theta^+(\fa_2 P^+(\Delta))$ with $H=A_1\Delta$, we obtain representatives $\fa_i$ and oriented bases $w_i'=Z_i(2A_i,B_i+\sqrt{\Delta})$ of $\fa_i$ such that
$$f_{w_i'}=f_i'=[A_i,B_i,C_i'],\qquad (A_1,\Delta)=1,\qquad (A_2,A_1\Delta)=1.$$
Then $\chisgn(\cN(Z_i)A_i)=1$, $(A_1A_2,\Delta)=1$ and $(A_1,A_2)=1$. 
The Chinese remainder theorem gives $B\in\cA$ such that $B\equiv B_i\pmod{2A_i}$ for $i=1,2$. 
Replacing $f_i'=[A_i,B_i,C_i']$ by the $\mathrm{SL}_2(\cA)$-equivalent form obtained from $f_i'(X+t_iY,Y)$, where $t_i=(B-B_i)/(2A_i)\in\cA$, we obtain that for $i=1,2$,
$w_i=Z_i(2A_i,B+\sqrt{\Delta})$ are oriented bases of $\fa_i$ and 
$f_{w_i}=f_i=[A_i,B,C_i]$.

Since $C_1,C_2\in\cA$ and $(A_1,A_2)=1$, we have $C_3\in\cA$. The form $f_3$ is primitive, because every divisor of $(A_1A_2,B,C_3)$ divides $\Delta$ and 
$(A_1A_2,\Delta)=1$.
Moreover, we have
$$\fa_1\fa_2
=Z_1Z_2[4A_1A_2,\,2A_1(B+\sqrt\Delta),\,2A_2(B+\sqrt\Delta),\,(B+\sqrt\Delta)^2]
=Z_1Z_2[2A_1A_2,B+\sqrt{\Delta}]$$
since $(A_1,A_2)=1$ and $(B+\sqrt{\Delta})^2=2B(B+\sqrt{\Delta})-4C_3A_1A_2$.
Thus $w_3=Z_1Z_2(2A_1A_2,B+\sqrt{\Delta})$ is an oriented basis of $\fa_1\fa_2$, and  $f_{w_3}=f_3$.
\end{proof}

\subsection{Continued fractions and real quadratic irrational functions}\label{app:continued-fractions}
Let $\PP^1=k_\infty\cup\{P_\infty\}$.  
For $\fM=\begin{pmatrix}A&B\\ C&D\end{pmatrix}\in\mathrm{GL}_2(\cA)$, 
the M\"obius action of $\fM$ on $\PP^1$ is
\begin{equation}\label{app-mobius}
        \fM Z=\frac{AZ+B}{CZ+D},
\end{equation}
with the usual conventions for $Z=P_\infty$ and $Z=-D/C$.

For $W\in k_\infty\setminus k$ and a sequence $\alpha=(\alpha_r)_{r\ge1}$ with $\alpha_r\in\FF_q^\times$, define $W_1^\alpha=W$ and
\begin{equation}\label{l-l-comparision}
W_r^\alpha=M_r^\alpha+\frac{\alpha_r}{W_{r+1}^\alpha},\qquad 
M_r^\alpha=\lfloor W_r^\alpha\rfloor .
\end{equation}
We call $W_r^\alpha$ and $M_r^\alpha$ the $r$-th complete and partial quotients 
of type $\alpha$ of $W$, respectively, and call $[M_1^\alpha,M_2^\alpha,\cdots]$ the continued fraction of type $\alpha$ of $W$. 
The continued fraction is called standard (resp., negative) if 
$\alpha_r=1$ (resp., $\alpha_r=-1$) for every $r\geq1$; its complete and partial quotients are denoted by $W_r$, $M_r$ (resp., $W_r^N$, $M_r^N$).

If it exists,
$$\partial_W:=\min\{n\ge1:W/W_{n+1}\in\FF_q^\times\}$$
is called the quasi-period of $W$.

Throughout this subsection, we assume that $\Delta$ is a monic nonsquare polynomial of even degree. Let  
\begin{equation*}
\XX_\Delta
=\left\{Z=\frac{B+\sqrt\Delta}{2A}: A,B,C\in\cA, \, A\neq0, \, (A,B,C)=1,\, B^2-4AC=\Delta\right\}
\end{equation*}
be the set of real quadratic irrational functions of discriminant $\Delta$, 
and $Z=(B+\sqrt\Delta)/(2A)$ is said to be of type $(A,B,C)$. Then its type is unique. 

We write $Z\sim Z'$ if $Z=\fM Z'$ for some $\fM\in\mathrm{GL}_2(\cA)$, and $Z\sim_+Z'$ if one may choose $\fM\in\mathrm{SL}_2(\cA)$.  

A real quadratic irrational function $Z\in\XX_\Delta$ is said to be reduced if $\|\overline Z\|<1<\|Z\|$.
Then one can show that the subset $\XX_\Delta^0$ of reduced elements is given by
\begin{equation}\label{finite-reduced}
\XX_\Delta^0
=\left\{\frac{B+\sqrt\Delta}{2A}\in\XX_\Delta: R\in\cA,\,
B=R+\lfloor\sqrt\Delta\rfloor,\,\deg R, \, \deg A, \, \deg C<\frac{\deg{\Delta}}{2}\right\}.
\end{equation}

\begin{prop}\label{app-real-correspondence}
Let $\Delta$ be a monic nonsquare polynomial of even degree. 
Choose $u\in\cO_\Delta$ with $\chisgn(\cN(u))=-1$ as in 
Proposition \ref{prop-negative-norm-element}. 
For $A\in\cA\setminus\{0\}$, put $\nu_A:=(1-\chisgn(A))/2$. 
The assignment 
$$\iota_\Delta : \XX_\Delta/\sim_+ \,\xrightarrow{\sim}\, C^+(\Delta),\qquad
[Z]_{\sim_+}\mapsto u^{\nu_A}[2A,B+\sqrt\Delta] P^+(\Delta),$$
where $Z=(B+\sqrt\Delta)/(2A)\in\XX_\Delta$, is a bijection. 
Moreover, there is a bijection 
$$\XX_\Delta/\sim_+\, \simeq\, \XX_\Delta^0/\sim_+.$$
Replacing $\mathrm{SL}_2(\cA)$ by $\mathrm{GL}_2(\cA)$ gives
$$\XX_\Delta/\sim\ \simeq\ \XX_\Delta^0/\sim\ \simeq\ C(\Delta).$$
\end{prop}

\begin{proof}
For a primitive form $f=[A,B,C]$ of discriminant $\Delta$ and 
$Z_f=(B+\sqrt{\Delta})/(2A)\in\XX_\Delta$,
the form action \eqref{proper-form-action} and the M\"obius action \eqref{app-mobius} are compatible:
$$Z_{\fM\cdot f}=\fM Z_f,\qquad \fM\in\mathrm{GL}_2(\cA).$$
Therefore Proposition \ref{form-ideal-correspondence-prop} identifies the $\sim_+$- and $\sim$-classes in $\XX_\Delta$ with the narrow and wide ideal classes. 
The basis $u^{\nu_A}(2A,B+\sqrt\Delta)$ is oriented:
$\chisgn\bigl(\Omega(u^{\nu_A}(2A,B+\sqrt\Delta))\bigr)
=\chisgn(4\cN(u)^{\nu_A}A)=1$.
This gives the displayed assignment.

It remains to compare arbitrary classes with reduced ones.  
For $W=W_1^N\in\XX_\Delta$ of type $(A_1,B_1,C_1)$, consider the negative continued fraction of $W$: put $M_n^N=\lfloor W_n^N\rfloor$ and
$$W_{n+1}^N=-\frac{1}{W_n^N-M_n^N}.$$
Then
$$W_{n+1}^N=\begin{pmatrix}0&1\\-1&M_n^N\end{pmatrix}W_{n}^N,\qquad
\det\begin{pmatrix}0&1\\-1&M_n^N\end{pmatrix}=1,$$
and we get $W_{n+1}^N\sim_+ W_n^N$.  
Let $W_{n}^N\in\XX_\Delta$ be of type $(A_{n},B_{n},C_{n})$. Then we have 
$$A_{n+1}=(M_n^{N})^2A_n-M_n^NB_n+C_n,\qquad
B_{n+1}=2M_n^NA_n-B_n,\qquad
C_{n+1}=A_n.$$
Since $W_{n+1}^N
=(2C_{n+1})/(B_{n+1}-\sqrt{\Delta})$ and $\deg{W_{n+1}^N}>0$, we have 
$\deg(B_{n+1}-\sqrt\Delta)<\deg A_n$.
If $W_{n+1}^N$ is not reduced, then $\deg\overline{W_{n+1}^N}\ge0$, and therefore
$$\deg A_{n+1} \le \deg(B_{n+1}-\sqrt\Delta).$$ 
Thus the degrees of $A_n$ strictly decrease until a reduced element is reached.  
Hence every $\sim_+$-class contains a reduced element, 
which gives a bijection between $\XX_\Delta/\sim_+$ and $\XX_\Delta^0/\sim_+$. 

The wide statement also follows.
\end{proof}

\begin{prop}\label{reduced-cf-periodicity}
For $W\in\XX_\Delta^0$, every standard complete quotient $W_r$ is reduced, the standard continued fraction of $W$ is purely periodic, and the quasi-period $\partial_W$ exists. Moreover, for $W\in\XX_\Delta$, the standard continued fraction of $W$ is eventually periodic.
\end{prop}
\begin{proof}
Put $\sS(W)=1/(W-\lfloor W\rfloor)$. Let $W\in\XX_\Delta^0$ and $M=\lfloor W\rfloor$. 
Since $\|W-M\|<1$, we have $\|\sS(W)\|>1$.
Since $\overline{\sS(W)}=1/(\overline{W}-M)$ and $\|\overline W\|<1<\|M\|$, 
the ultrametric inequality gives $\|\overline{W}-M\|=\|M\|>1$, 
and hence $\|\overline{\sS(W)}\|<1$.
Thus $\sS(W)$ is reduced, and $\sS$ preserves $\XX_\Delta^0$.

Suppose $V,W\in\XX_\Delta^0$ and $\sS(V)=W$. Put $M=\lfloor V \rfloor$.
Taking conjugates in $V=M+1/W$ and using the reducedness of $V$ gives
$M=-\lfloor 1/\overline{W}\rfloor$ and $V=-\lfloor 1/\overline{W}\rfloor+1/W$; 
so, $V$ is uniquely determined by $W$. 
Hence $\sS:\XX_\Delta^0\to\XX_\Delta^0$ is injective. 
Since $\XX_\Delta^0$ is finite by \eqref{finite-reduced}, $\sS$ is a permutation of $\XX_\Delta^0$, and
every orbit of $\sS$ on $\XX_\Delta^0$ is purely periodic. 

Let $W\in\XX_\Delta$. We denote by $W_r$ and $W_r^N$ the $r$-th standard and negative complete quotients of $W$. 
A direct induction gives $W_r=(-1)^{r-1}W_r^N$ for $r\ge1$.
The proof of Proposition \ref{app-real-correspondence} shows that $W_r^N$ is reduced for some $r$; so is $W_r$. 
Then the preceding argument gives eventual periodicity of the standard continued fraction of $W$.
\end{proof}

\begin{prop}\label{last-remark-thm-1.4}
For each class in $\XX_\Delta^0/\sim$, one may choose a reduced
representative $W=(B+\sqrt{\Delta})/(2A)$ satisfying $\sgn\overline W=1$.  For such a
representative, $\sgn(AW)=1$. 
\end{prop}

\begin{proof}
Multiplying $W\in\XX_\Delta^0$ by $\lambda\in\FF_q^\times$ preserves
reducedness and is the action of
$\begin{pmatrix}\lambda&0\\0&1\end{pmatrix}\in\mathrm{GL}_2(\cA)$.  
Therefore we may assume that
$\sgn\overline W=1$.  Since $W$ is reduced,
$\deg(B-\sqrt\Delta)<\deg(B+\sqrt\Delta)$, and hence
$\deg B=\deg\sqrt\Delta$ and $\sgn B=\sgn\sqrt\Delta=1$. 
Thus
$\sgn(AW)=\sgn\left(\frac{B+\sqrt\Delta}{2}\right)=1$.
\end{proof}

\section{Genus characters}\label{app:genus-characters}
We now give the form-theoretic construction of the genus character. 
Recall that for  $M\in\cA\setminus\{0\}$ and $N\in\cA$,
$$J_M(N):=\left[\frac{N}{M}\right].$$
For a primitive quadratic form $f=[A,B,C]$ of discriminant $\Delta\in\cA\setminus\cA^2$ and a prime $P\mid\Delta$, define
\begin{equation}\label{local-genus-form}
J^{(P)}(f):=\begin{cases}
J_P(A),&P\nmid A,\\
J_P(C),&P\mid A.\end{cases}
\end{equation}
If $P\mid A$, then $P\mid B$ and primitivity gives $P\nmid C$;
hence $J^{(P)}(f)\in\{\pm1\}$ for every $P\mid\Delta$.

\begin{lem}\label{represented-value-local}
Let $m=f(x,y)\in\cA$ with $x,y\in\cA$ and $(m,P)=1$. Then $J^{(P)}(f)=J_P(m)$.
\end{lem}
\begin{proof}
Let $f=[A,B,C]$. If $P\nmid A$, then $(2Ax+By)^2\equiv4Af(x,y)\not\equiv0\pmod P$, and  hence $J_P(A)=J_P(m)$. 
If $P\mid A$, apply the same argument to
$(Bx+2Cy)^2\equiv4Cf(x,y)\pmod P$.
\end{proof}

Let $D_1,D_2\in\cA$ be square-free, let $F_0\in\cA^{(1)}$, 
and assume $\Delta=D_1D_2F_0^2\in\cA\setminus\cA^2$. Define
$$\chi_{D_1,D_2}^{(\Delta)}(f):=\prod_{P\mid D_1}J^{(P)}(f).$$
It follows from Lemma \ref{represented-value-local} that 
if $m$ is represented by $f$ and $(m,D_1)=1$, 
\begin{equation}\label{signed-genus-reciprocal-value}
\chi_{D_1,D_2}^{(\Delta)}(f)=J_{D_1}(m)=\left[\frac{m}{D_1}\right].
\end{equation}

\begin{prop}\label{genus-character-homomorphism}
The function $f \mapsto \chi_{D_1,D_2}^{(\Delta)}(f)$ on the set of all primitive quadratic forms of discriminant $\Delta$ is invariant under $\mathrm{SL}_2(\cA)$-equivalence. 
Furthermore, via the bijection between $\mathcal{F}_\Delta^+$ and $C^+(\Delta)$ 
in Proposition \ref{form-ideal-correspondence-prop}, 
it induces a character
$$\chi_{D_1,D_2}^{(\Delta)} : C^+(\Delta)\rightarrow\{\pm1\}.$$
\end{prop}
\begin{proof}
Let $f'=\fM\cdot f$ with $\fM\in\mathrm{SL}_2(\cA)$.  
By Lemma \ref{qf-coefficient-prime}, the primitive form $f'$ represents some $m\in\cA$ 
such that $(m,D_1)=1$.  
Then $(\det{\fM})m$ is represented by $f$, and 
$J_{D_1}(\det{\fM})=\chisgn(\det{\fM})^{\deg{D_1}}=1$.
Therefore, by \eqref{signed-genus-reciprocal-value},
$$\chi_{D_1,D_2}^{(\Delta)}(f')=J_{D_1}(m)
=J_{D_1}((\det{\fM})m)=\chi_{D_1,D_2}^{(\Delta)}(f).$$
Thus $\chi_{D_1,D_2}^{(\Delta)}$ is invariant under $\mathrm{SL}_2(\cA)$-equivalence and
induces a well-defined function on $\mathcal F_\Delta^+$ and hence,
via Proposition \ref{form-ideal-correspondence-prop}, on $C^+(\Delta)$.

Let $\fa_1 P^+(\Delta)$, $\fa_2 P^+(\Delta)\in C^+(\Delta)$. By Proposition \ref{prime-to-delta-representative}, 
we may choose representatives $\mathfrak a_i$ and oriented bases $w_i$ ($i=1$, $2$)
such that $f_{w_i}=[A_i,B,C_i]$ with $(A_1A_2,\Delta)=1$ and $(A_1,A_2)=1$; 
$\fa_1\fa_2$ has an oriented basis $w_3$ satisfying $f_{w_3}=[A_1A_2,B,C_3]$. 
By \eqref{signed-genus-reciprocal-value}, we have 
$$\chi_{D_1,D_2}^{(\Delta)}(\fa_1 P^+(\Delta))
\chi_{D_1,D_2}^{(\Delta)}(\fa_2 P^+(\Delta))
=J_{D_1}(A_1)J_{D_1}(A_2)
=J_{D_1}(A_1A_2)
=\chi_{D_1,D_2}^{(\Delta)}(\fa_1 \fa_2 P^+(\Delta)).$$
Hence $\chi_{D_1,D_2}^{(\Delta)}$ is a group homomorphism.
\end{proof}

We call $\chi_{D_1,D_2}^{(\Delta)}:C^+(\Delta)\to\{\pm1\}$ the genus character. 
We also regard this as a character on $\mathcal F_\Delta^+$
via Proposition \ref{form-ideal-correspondence-prop}, 
and when $\Delta$ is monic, nonsquare, and of even degree, as a character on
$\XX_\Delta/\sim_+$ via Proposition \ref{app-real-correspondence}.

\begin{prop}\label{symmetry-genus-characters}
Let $f=[A,B,C]$ be a primitive quadratic form of discriminant $\Delta$.  Then
\begin{equation}\label{general-symmetry-obstruction}
\chi_{D_1,D_2}^{(\Delta)}(f)\chi_{D_2,D_1}^{(\Delta)}(f)
=\chisgn(-1)^{\deg\Delta\cdot\deg A} \chisgn(\Delta)^{\deg A} \chisgn(A)^{\deg\Delta}.
\end{equation}
In particular, if $\Delta$ is monic of even degree, then
$\chi_{D_1,D_2}^{(\Delta)}=\chi_{D_2,D_1}^{(\Delta)}$.
\end{prop}
\begin{proof}
By Lemma \ref{qf-coefficient-prime}, we may choose $x_0\in\cA$ such that $(f(x_0,1),\Delta)=1$. 
For $H\in\cA$, 
put $x=x_0+H\Delta$ and $m=f(x,1)$.  
For $H$ of sufficiently large degree, the term $Ax^2$ dominates $Bx+C$; hence
\begin{equation*}\label{replace-order-eq}
\deg m\equiv\deg A\pmod2,\qquad \chisgn(m)=\chisgn(A).
\end{equation*}
Since $(m,\Delta)=1$, by \eqref{signed-genus-reciprocal-value}, the left-hand side of
\eqref{general-symmetry-obstruction} is $J_\Delta(m)$. 
Since $4Am=(2Ax+B)^2-\Delta$, we have $J_m(\Delta)=1$; 
by the quadratic reciprocity law, 
$$J_\Delta(m)=\chisgn(-1)^{\deg\Delta \cdot \deg m}
\chisgn(\Delta)^{\deg m} \chisgn(m)^{\deg \Delta}.$$
Therefore the identity \eqref{general-symmetry-obstruction} follows.
\end{proof}

\begin{prop}\label{change-Rit-11.3.3} 
Let $\Delta=D_1D_2F_0^2$. 
Assume that either $\deg \Delta $ is even, or $\deg \Delta $ is odd and $\chisgn(-\Delta)=-1$. 
Let $u\in\cO_\Delta$ satisfy $\chisgn(\cN(u))=-1$ 
as in Proposition \ref{prop-negative-norm-element}. 
Then for every $\fa P^+(\Delta)\in C^+(\Delta)$, 
$$\chi_{D_1,D_2}^{(\Delta)}(u\fa P^+(\Delta))
=(-1)^{\deg{D_1}}\chi_{D_1,D_2}^{(\Delta)}(\fa P^+(\Delta)).$$
Furthermore, if $\deg D_1$ is even, then $\chi_{D_1,D_2}^{(\Delta)}$ descends to $C(\Delta)$. If $\deg D_1$ is odd, then $h^+(\Delta)=2h(\Delta)$ and $\chi_{D_1,D_2}^{(\Delta)}$ does not descend to $C(\Delta)$.
\end{prop}
\begin{proof}
Applying Lemma \ref{qf-coefficient-prime} to $\Theta^+(\fa P^+(\Delta))$ with $H=D_1$, 
one may choose a representative $\fa$ and an oriented basis $w=(w_1,w_2)$ of $\fa$ 
such that 
$$f_{w}=[A,B,C],\qquad (A,D_1)=1.$$
Put
$\widetilde{w}=(guw_1,uw_2)$ and $\fM=\begin{pmatrix}1&0\\0&g\end{pmatrix}$.
The basis transformation formulas \eqref{form-action-properties} give
$$\Omega(\widetilde{w})=g\cN(u)\Omega(w),\qquad 
f_{\widetilde{w}}=\fM\cdot f_w=[gA,B,g^{-1}C].$$ 
Since $\chisgn(g)=\chisgn(\cN(u))=-1$, $\widetilde{w}$ is an oriented basis of $u\fa$. 
Hence by \eqref{signed-genus-reciprocal-value},
$$\chi_{D_1,D_2}^{(\Delta)}(u\fa P^+(\Delta))
=\left[\frac{gA}{D_1}\right]
=(-1)^{\deg D_1}\left[\frac{A}{D_1}\right]
=(-1)^{\deg D_1}\chi_{D_1,D_2}^{(\Delta)}(\fa P^+(\Delta)).$$
Recall that
$\ker\left(C^+(\Delta) \to C(\Delta)\right)=P(\Delta)/P^+(\Delta)
=\{P^+(\Delta),\,u\cO_\Delta P^+(\Delta)\}$, 
where the two narrow classes may coincide. 
Taking $\fa=\cO_\Delta$ gives
$$\chi_{D_1,D_2}^{(\Delta)}(\cO_\Delta P^+(\Delta))=1,\qquad
\chi_{D_1,D_2}^{(\Delta)}(u\cO_\Delta P^+(\Delta))=(-1)^{\deg D_1}.$$
If $\deg D_1$ is even, the character is trivial on this kernel
and then descends to $C(\Delta)$.
If $\deg D_1$ is odd, it takes the value $-1$ on $u\cO_\Delta P^+(\Delta)$; 
thus the kernel has order $2$, $h^+(\Delta)=2h(\Delta)$, and 
the character does not descend to $C(\Delta)$.
\end{proof}

\end{document}